\documentclass{siamart220329}
\usepackage{graphicx}
\usepackage{diagbox}
\usepackage{standalone}
\usepackage{amsmath}								
\usepackage{amssymb}
\usepackage{amsfonts}
\usepackage{mathtools}
\usepackage{bm}
\usepackage{bbm}
\usepackage{indentfirst}
\usepackage{enumitem}
\usepackage{algpseudocode}
\usepackage{algorithm}
\usepackage{multirow}
\usepackage{makecell}
\usepackage{subfigure}
\usepackage{hyperref}
\usepackage{xcolor}

\usepackage{tikz}
\usetikzlibrary{calc}
\usetikzlibrary{arrows.meta}
\tikzset{
  dot/.style={
    circle, fill=black, inner sep=1pt, outer sep=0pt
  },
  dot label/.style={
    circle, inner sep=0pt, outer sep=1pt
  }
  arrow1/.style = {
    draw = black, thick, -{Latex[length = 4mm, width = 1.5mm]},
  }
}

\newcommand{\ibold}{\mathbf{i}}
\newcommand{\jbold}{\mathbf{j}}
\newcommand{\ebold}{\bm{\mathrm{e}}}
\newcommand{\Dim}{{\scriptsize \textsf{D}}}

\newcommand{\avg}[1]{\left\langle #1 \right\rangle}
\newcommand{\bbR}{\mathbb{R}}
\newcommand{\bbY}{\mathbb{Y}}
\newcommand{\bbZ}{\mathbb{Z}}

\newcommand{\bmi}{\mathbf{i}}

\newcommand{\bmihalf}{\bmi + \frac 1 2 \bm{\mathrm{e}}^d}
\newcommand{\bmj}{\mathbf{j}}

\newcommand{\bmn}{\mathbf{n}}
\newcommand{\bmp}{\mathbf{p}}

\newcommand{\bmx}{\mathbf{x}}

\newcommand{\bmalpha}{\bm{\alpha}}

\newcommand{\calB}{\mathcal{B}}
\newcommand{\calC}{\mathcal{C}}

\newcommand{\calF}{\mathcal{F}}

\newcommand{\calL}{\mathcal{L}}

\newcommand{\calN}{\mathcal{N}}
\newcommand{\calP}{\mathcal{P}}
\newcommand{\calS}{\mathcal{S}}
\newcommand{\calT}{\mathcal{T}}
\newcommand{\calX}{\mathcal{X}}
\newcommand{\calY}{\mathcal{Y}}
\newcommand{\cuppp}{\cup^{\perp \perp}}

\DeclareMathOperator{\diag}{diag}

\DeclareMathOperator{\bigO}{O}

\DeclareMathOperator*{\argmax}{argmax}
\DeclareMathOperator{\intr}{int}
\newcommand{\basis}{\bm{\mathrm{e}}}
\newcommand{\Span}{\mathrm{span}}
\newcommand{\rmd}{\mathrm{d}}
\newcommand{\rmF}{\mathrm{F}}
\newcommand{\bbone}{\mathbbm{1}}
\newcommand{\brk}[1]{\left\langle #1 \right\rangle}

\newsiamremark{remark}{Remark}

\makeatletter
\newcommand{\fdsy@scale}{1.0}
\newcommand\fdsy@mweight@normal{Book}%
\newcommand\fdsy@mweight@small{Regular}%
\newcommand\fdsy@bweight@normal{Medium}%
\newcommand\fdsy@bweight@small{Bold}%
\DeclareFontFamily{U}{FdSymbolF}{}
\DeclareFontShape{U}{FdSymbolF}{m}{n}{
    <-7.1> s * [\fdsy@scale] FdSymbolF-\fdsy@mweight@small
    <7.1-> s * [\fdsy@scale] FdSymbolF-\fdsy@mweight@normal
}{}
\DeclareFontShape{U}{FdSymbolF}{b}{n}{
    <-7.1> s * [\fdsy@scale] FdSymbolF-\fdsy@bweight@small
    <7.1-> s * [\fdsy@scale] FdSymbolF-\fdsy@bweight@normal
}{}
\makeatother

\DeclareSymbolFont{delimiters}{U}{FdSymbolF}{m}{n}
\SetSymbolFont{delimiters}{bold}{U}{FdSymbolF}{b}{n}
\DeclareMathDelimiter{\lAngle}{\mathopen}{delimiters}{"92}{delimiters}{"92}
\DeclareMathDelimiter{\rAngle}{\mathclose}{delimiters}{"98}{delimiters}{"92}
\newcommand{\bbrk}[1]{\left\lAngle #1 \right\rAngle}

\title{A Fourth-Order Cut-cell Multigrid Method \\
  for Solving Elliptic Equations on Arbitrary Domains\thanks{Qinghai Zhang
    is the corresponding author (\email{qinghai@zju.edu.cn}).
    Jiyu Liu and Zhixuan Li are the co-first authors
    with equal contributions. 
    \funding{This work was supported by
      the Fundamental Research Funds
      for the Central Universities 226-2025-00254
      and the National Natural Science Foundation of China (\#12272346)}}
}

\author{Jiyu Liu\thanks{School of Mathematical Sciences,
    Zhejiang University, Hangzhou, Zhejiang, 310058, China.}
  \and Zhixuan Li\footnotemark[2]
  \and Jiatu Yan\footnotemark[2]
  \and Zhiqi Li\footnotemark[2]
  \and Qinghai Zhang\footnotemark[1] \footnotemark[2]
  \thanks{Institute of Fundamental and Transdisciplinary Research,
    Zhejiang University, Hangzhou, Zhejiang, 310058,
    China.}
}

\headers{A Cut-Cell Multigrid Method
  for Elliptic Equations}{J. Liu, Z. Li, J. Yan, Z. Li, and Q. Zhang}

\begin{document}
\maketitle

\begin{abstract}
To numerically solve a generic elliptic equation
 on two-dimensional domains with rectangular Cartesian grids, 
 we propose a cut-cell geometric multigrid method
 that features
 (1) general algorithmic steps
 that apply to two-dimensional constant-coefficient elliptic equations
 with both divergence and non-divergence forms
 and all types of boundary conditions, 
 (2) the versatility of
 handling both regular and irregular domains
 with arbitrarily complex topology and geometry, 
 (3) the fourth-order accuracy
 even at the presence of ${\cal C}^1$ discontinuities on the domain boundary, 
 and (4) the optimal complexity of $O(h^{-2})$.
Test results 
  demonstrate the generality, accuracy, efficiency, robustness,
  and excellent conditioning of the proposed method.


\end{abstract}

\begin{keywords}
  elliptic equations,
  finite-volume methods, 
  geometric multigrid methods,
  poised lattice generation, 
  generating cut cells without small volumes. 
\end{keywords}
\begin{AMS}
{65N08, 65N55}
\end{AMS}

\section{Introduction}
\label{sec:intro}

Consider the constant-coefficient elliptic equation
\begin{subequations}
  \label{eq:ccEllipticEq}
  \begin{alignat}{2}
    {\cal L} u &:= a \frac{\partial^2 u}{\partial x^2}
    + b \frac{\partial^2 u}{\partial x \partial y}
    + c \frac{\partial^2 u}{\partial y^2}
    = f(x,y)  \quad &&\text{in } \Omega,
    \\
    \calN u &= g(x,y)  &&\text{on }\partial \Omega.
  \end{alignat}
\end{subequations}
where the domain $\Omega$ is an open subset of $\bbR^2$, 
$u$ the unknown function,
$a, b, c$ real numbers satisfying $b^2 - 4ac < 0$, 
and $\calN$ the boundary-condition operator:
\begin{equation}
  \label{eq:boundaryCondOp}
  \calN =
  \begin{cases}
    \mathbf{I}_\text{d}
    & \text{for the Dirichlet condition $u = g$},
    \\
    \frac{\partial}{\partial \bmn}
    & \text{for the Neumann condition $\bmn \cdot \nabla u = g$},
    \\
    \alpha_1 + \alpha_2 \frac{\partial}{\partial \bmn}
    & \text{for the Robin condition $\alpha_1 u + \alpha_2\bmn \cdot \nabla u = g$},
  \end{cases}
\end{equation}
with $\alpha_1, \alpha_2 \in \bbR$
and $\mathbf{I}_\text{d}$ as the identity operator.

Set $(a, b, c) = (1, 0, 1)$ in (\ref{eq:ccEllipticEq})
 and we get Poisson's equation, 
 which is particularly important for designing numerical methods
 for solving partial differential equations (PDEs). 
For example,
at the core of the projection methods
 \cite{Brown2001, Johnston2004, Liu07, Zhang2014, Zhang2016:GePUP, li25:_gepup_es}
 for the incompressible Navier--Stokes equations (INSE) 
 is the numerical solution of a sequence of Poisson's equations
 and Helmholtz-like equations.
 
A myriad of numerical methods have been developed
 for solving PDEs on regular domains. 
In many real-world applications, however, 
 a problem domain may have complex topology and irregular geometry. 
Within the realm of finite element methods (FEMs), 
 the irregular geometry is handled
 either by an external mesh generator, 
 yielding the interface-fitted FEMs
 \cite{Babuska1970, BARRETT1987, Chen1998, Mu2013}, 
 or by local treatments of the irregular boundary
 inside its algorithms, 
 leading to the interface-unfitted FEMs~\cite{Li1998253, Gong08, Guo2023}. 
As for finite difference (FD) methods,
 two notable examples 
 are the immersed interface method~\cite{LeVeque1994, Li1998, Li2003, Linnick2005}, 
 where the discretization stencil is modified
 to incorporate jump conditions on the interface,
 and the ghost fluid method~\cite{Gibou2002, Liu2003, Liu2000, LiuMGFM2023}, 
 where physical variables are smoothly extended across the interface
 so that conventional FD formulas can be used for spatial discretization. 
For other FD methods on irregular domains,
 see \cite{Zhou2006, HOSSEINVERDI2018, Colnago2020, 
   Rapaka2020, Zhang:PLG}
 and references therein.

For finite volume (FV) formulations, 
 the cut cell method,
 also known as the embedded boundary (EB) method,
 consists of three main steps as follows.
 \begin{enumerate}[label={(CCM-\arabic*)}, leftmargin=*]
 \item Use a Cartesian grid to partition the domain $\Omega$ 
   into a set of cut cells,
   which are irregular near an irregular boundary
   and regular otherwise.
 \item Approximate the average of ${\cal L} u$ 
   over each cut cell by a linear expression
   of averages of $u$ over nearby cut cells, 
   cf. (\ref{eq:cellAvg}), 
   and consequently discretize (\ref{eq:ccEllipticEq})
   into a system of linear equations.
 \item Solve the linear system to obtain
   averages of $u$ over all cut cells
   as the numerical solution of (\ref{eq:ccEllipticEq}). 
 \end{enumerate}

For Poisson's equation in two-dimensions (2D), 
 Johansen and Colella~\cite{Johansen1998}
 proposed a second-order EB method, 
 in which  
 the average of $\nabla\cdot(\beta\nabla u)$ 
 over a cut cell
 is transformed by the divergence theorem
 into a sum of fluxes through cell faces. 
The fluxes through regular faces are approximated by standard FV formulas 
while those near an irregular domain boundary
 by quadratic interpolations. 
This second-order EB method has been extended to 
 the three-dimensional Poisson's equation~\cite{Schwartz2006}, 
 the heat equation~\cite{McCorquodale2001, Schwartz2006}, 
 and the INSE~\cite{Kirkpatrick2003, Trebotich15}.
To improve the second-order accuracy to the fourth order, 
 one must successfully address three main difficulties,
 viz. the representation of irregular geometry, 
 the approximation of integrals over irregular cut cells,
 and the discretization of (\ref{eq:ccEllipticEq}) 
 with sufficient accuracy.

As far as fourth-order FV methods
 on irregular domains are concerned, 
 we are only aware of the high-order EB method
 developed by Colella~\cite{Colella2016}
 and colleagues~\cite{Devendran2017}, 
 in which the irregular domain is represented
 as $\Omega=\{\mathbf{x}: \phi(\mathbf{x})<0\}$
 with $\phi$ being a smooth level set function
 $\mathbb{R}^2\rightarrow \mathbb{R}$,
 the domain boundary $\partial \Omega$ as $\phi^{-1}(0)$, 
 and the unit normal vector of $\partial \Omega$
 as  $\mathbf{n}=\frac{\nabla \phi}{\|\nabla \phi\|}$.
Assuming the existence of a flux vector $\mathbf{F}(u)$ 
 satisfying ${\cal L}u = \nabla\cdot \mathbf{F}(u)$, 
 they use the divergence theorem
 to transform the integral of ${\cal L} u$
 over an irregular cut cell to 
 those of $\mathbf{F}(u)$ over the cell faces. 
To further discretize (\ref{eq:ccEllipticEq}) on a cut cell,
 they expand $\mathbf{F}(u)$ in its Taylor series,
 fit a local multivariate polynomial
 via weighted least squares~\cite[\S 2.3]{Devendran2017}, 
 calculate moments of monomials 
 over the cell faces,
 and finally approximate the integral of $\nabla\cdot \mathbf{F}(u)$
 over the cut cell with
 a linear combination of integrals of $\mathbf{F}(u)$
 over nearby cell faces. 
As such, the approximation of the irregular geometry
 and the discretization of (\ref{eq:ccEllipticEq})
 are tightly coupled.
 
In spite of its successes,
 the aforementioned fourth-order EB method~\cite{Colella2016,Devendran2017}  
 has a number of limitations.
First, the assumption of the divergence form
 \mbox{${\cal L}u = \nabla\cdot \mathbf{F}(u)$} 
 limits the generality of the EB method, 
 since many elliptic equations
 with Neumann conditions
 and a cross-derivative term have no divergence form. 
Second,
 the representation of irregular geometry
 by a level set function $\phi$
 leads to severe accuracy deterioration
 at the presence of kinks, i.e., ${\cal C}^1$ discontinuities,
 at which $\frac{\nabla \phi}{\|\nabla \phi\|}$
 has an $O(1)$ error in approximating
 the normal vector $\mathbf{n}$.
Indeed,  Devendran et al.~\cite{Devendran2017}
 reported that the accuracy of their fourth-order EB method
 drops to the first order or non-convergence 
 in solving a Poisson equation
 on a domain with kinks, cf. \Cref{tab:squareMinusFourDisks}. 
Although this accuracy deterioration can be alleviated
 by mollifying the kinks,
 this mollification is not applicable to all cases, 
 and its effectiveness
 depends heavily on the nature of the equation and
 the mollification formula~\cite{Devendran2017}. 
Third,  
 there is no guarantee that
 the lattices (or stencils) for weighted least squares
 such as that in \cite[Sec. 2.3.3]{Devendran2017}
 work well for all geometry. 
On one hand, 
 a wide stencil may lead to a large number of redundant points,
 adversely affecting computational efficiency.
Then it is desirable to have a poised lattice
 whose cardinality equals the dimension of the space
 of multivariate interpolating polynomials \cite{Zhang:PLG}. 
On the other hand, 
 the local geometry might make it impossible
 to fit a high-order multivariate polynomial
 out of nearby cut cells.
In this case, one needs to know the highest degree
 of interpolating polynomials for which
 the local geometry admits. 
As far as we know,
 the only algorithm that meets these requirements
 is that of the poised lattice generation (PLG) \cite{Zhang:PLG}
 for FD methods.

The above discussions pertain to (CCM-1) and (CCM-2).
As for (CCM-3),
 the linear system that results from discretizing
 the elliptic equation is typically solved 
 by a geometric multigrid method on regular domains
 and by an algebraic multigrid method 
 on irregular domains~\cite{Briggs:A_Multigrid_Tutorial}. 
Some researchers~\cite{Devendran2017, Johansen1998, Schwartz2006} 
 extend geometric multigrid methods to irregular boundaries
 by modifying the restriction and interpolation operators 
 on cells near irregular boundaries; 
 but it is not clear whether
 these multigrid methods can achieve the optimal complexity of
 $O(h^{-2})$.
It is observed in~\cite{Trebotich15}
 that these methods struggle with convergence
 on domains with very complex geometry.

The above discussions lead to questions as follows. 
\begin{enumerate}[label={(Q-\arabic*)}, leftmargin=*]
\item To represent an irregular domain $\Omega$
  with arbitrary geometry and topology, 
  can we have a simple and efficient scheme 
  that is always fourth-order accurate?
  Furthermore, for a given threshold $\epsilon\in(0,1)$
  and a Cartesian grid of size $h$, 
  can we partition $\Omega$ into a set of cut cells
  whose volumes are between $\epsilon h^2$ and $2 h^2$?
\label{q:DomainRepAndCutCell}
\item Can we design an FV discretization
  of (\ref{eq:ccEllipticEq}) on these cut cells
  so that (i) the discretization processes
  depends neither on values of $(a,b,c)$ in (\ref{eq:ccEllipticEq}a)
  nor on forms of boundary conditions in (\ref{eq:ccEllipticEq}b), 
  and (ii) the fourth-order accuracy 
  depends neither on the topology of $\Omega$
  nor on the \emph{absence of kinks} on $\partial\Omega$? 
\label{q:FV-discretization}

\item For the FV discretization in \ref{q:FV-discretization}, 
  can we further develop a geometric multigrid method
  that solves the resulting linear system  
  with \emph{optimal} complexity? 
\label{q:efficiency}
\end{enumerate}

In this paper, we give positive answers to all above questions
 by proposing a cut-cell geometric multigrid method
 for solving (\ref{eq:ccEllipticEq})
 over arbitrary 2D domains.

Our answer to \ref{q:DomainRepAndCutCell} is based 
 on Yin sets~\cite{Zhang2020:YinSets},
 a mathematical model of 2D continua with arbitrary topology, 
 which we briefly review in \Cref{sec:yin-space}. 
Utilizing the Boolean algebra of Yin sets in \cite{Zhang2020:YinSets}, 
 we propose a cut-cell algorithm in \Cref{sec:doma-rep}
 to generate a set $\mathsf{C}_{\epsilon}^h$ of cut cells
 so that the regularized union of these cut cells
 equals $\Omega$
 and the volume of each cut cell is no less than $\epsilon h^2$,
 precluding the well-known small-volume problem in FV methods. 

\ref{q:FV-discretization} is answered in \Cref{sec:discretization}. 
We call a cut cell a \emph{symmetric finite volume (SFV) cell}
 if classical symmetric FV formulas apply to it; 
 otherwise it is called a PLG cell, cf. \Cref{def:SFVCell}. 
The fourth-order discretization of the integral of (\ref{eq:ccEllipticEq})
 over SFV cells is given in \Cref{sec:symmetricFD-cells}
 and that over PLG cells
 is based on the PLG algorithm~\cite{Zhang:PLG}
 summarized in \Cref{sec:PLG}. 
Given $K\subset\bbZ^\Dim$,
 a starting point $q\in K$,
 and the total degree $n$ of $\Dim$-variate polynomials,
 this PLG algorithm generates a poised lattice
 on which $\Dim$-variate polynomial interpolation is unisolvent.
In \Cref{sec:Discretization-PLG-cells}, 
 this PLG algorithm is adapted to the FV formulation
 to generate linear equations
 that approximates integrals of (\ref{eq:ccEllipticEq})
 over PLG cells to sufficient accuracy. 
The complete linear system 
 is summarized in \Cref{sec:discreteElliptic}.

In \Cref{sec:multigrid}, 
 we answer \ref{q:efficiency}
 by proposing a cut-cell geometric multigrid method, 
 which hinges on the fact that numbers of PLG and SFV cells
 are $O(h^{-1})$ and $O(h^{-2})$, respectively;
 see the opening paragraph of \Cref{sec:multigrid}
 for other key ideas. 
In \Cref{sec:Tests},
 we demonstrate the accuracy, efficiency, generality,
 robustness, and excellent conditioning 
 of the proposed cut-cell method
 by results of a number of numerical tests.
We conclude this work in \Cref{sec:conclusions}
 with several future research prospects.



\section{Modeling continua with Yin sets}
\label{sec:yin-space}

In this section,
 we briefly review Yin sets \cite{Zhang2020:YinSets} 
 as a model of topological structures and geometric features
 of 2D continua.

In a topological space ${\mathcal X}$,
 the \emph{complement} of a subset ${\mathcal P}\subseteq {\mathcal X}$,
 written ${\mathcal P}'$,
 is the set ${\mathcal X}\setminus {\mathcal P}$.
The \emph{closure} of a set ${\mathcal P}\subseteq{\mathcal X}$,
 written $\overline{\mathcal P}$,
 is the intersection of all closed 
 supersets of ${\mathcal P}$.
The \emph{interior} of ${\mathcal P}$, written ${\mathcal P}^{\circ}$,
 is the union of all open subsets of ${\mathcal P}$.
The \emph{exterior} of ${\mathcal P}$,
 written ${\mathcal P}^{\perp}:= {\mathcal P}^{\prime\circ}
 :=({\mathcal P}')^{\circ}$,
 is the interior of its complement.
A point $\mathbf{x}\in {\mathcal X}$ is
 a \emph{boundary point} of ${\mathcal P}$
 if $\mathbf{x}\not\in {\mathcal P}^{\circ}$
 and $\mathbf{x}\not\in {\mathcal P}^{\perp}$.
The \emph{boundary} of ${\mathcal P}$, written $\partial {\mathcal P}$,
 is the set of all boundary points of ${\mathcal P}$.

A subset $\calP$ in ${\cal X}$ is \emph{regular open}
 if it coincides with the interior of its closure.
For $\calX = \bbR^2$, 
a subset $\calS \subset \bbR^2$ is \emph{semianalytic}
if there exist a finite number of analytic functions
$g_i : \bbR^2 \rightarrow \bbR$
such that $\calS$ is in the universe of a finite Boolean algebra
formed from the sets
$\calX_i = \left\{ \bmx \in \bbR^2 : g_i(\bmx) \ge 0 \right\}$. 
In particular, a semianalytic set is \emph{semialgebraic}
if all the $g_i$'s are polynomials. 
These concepts lead to 

\begin{definition}[Yin Space~\cite{Zhang2020:YinSets}]
  \label{def:yin-space}
  A \emph{Yin set} $\calY \subseteq \bbR^2$
   is a regular open semianalytic set 
   with bounded boundary.
  The Yin space $\bbY$ is
   the class of all Yin sets. 
\end{definition}

In \Cref{def:yin-space}, 
 the regularity captures the salient feature
 that continua are free of low-dimensional elements
 such as isolated points and crevices,  
 the openness leads to a unique boundary representation of any Yin set, 
 and the semianalyticity ensures that a finite number of entities
 suffice for the boundary representation.

Each Yin set $\calY \neq \emptyset, \bbR^2$
 can be uniquely expressed \cite[Corollary 3.13]{Zhang2020:YinSets} as
 \begin{equation}
   \label{eq:uniqueRepOfYinSets}
   \calY = \cup^{\perp\perp}_{j} {\cal Y}_j
   = \cup^{\perp\perp}_{j}
   \cap_{i} \intr(\gamma_{j, i}),
 \end{equation}
where ${\cal Y}_j$ is the $j$th connected component of $\calY$,
the binary operation $\cuppp$ the \emph{regularized union} defined as
$\calS \cuppp \calT := \left(\calS \cup \calT \right)^{\perp \perp}$, 
$\{\gamma_{j,i}\}$ 
the set of pairwise almost disjoint Jordan curves
satisfying $\partial {\cal Y}_j=\cup_i\gamma_{j,i}$, 
and $\intr(\gamma_{j,i})$ the complement of $\gamma_{j,i}$
that always lies at the left of an observer
who traverses $\gamma_{j,i}$ according to its orientation.

\begin{theorem}[Zhang and Li~\cite{Zhang2020:YinSets}]
  \label{thm:BooleanAlgebra}
  $\left(\bbY, \cuppp, \cap, ^\perp, \emptyset, \bbR^2 \right)$
  is a Boolean algebra.
\end{theorem}

\begin{corollary}
  \label{coro:BooleanAlgebra}
  Denote by $\bbY_{\mathrm{c}}$ the subspace of $\bbY$
  where the boundary of each Yin set 
  is constituted by a finite number of cubic splines.
  Then $\bbY_{\mathrm{c}}$ is a sub-algebra of $\bbY$. 
\end{corollary}

The above Boolean algebra is implemented
 in \cite{Zhang2020:YinSets}.
In this work, the problem domain $\Omega$ in (\ref{eq:ccEllipticEq})
 is assumed to be in $\mathbb{Y}$
 and approximated by a Yin set in $\bbY_\mathrm{c}$.


\section{Partitioning a Yin set into cut cells}
\label{sec:doma-rep}

The arbitrary complexity of $\Omega$ is handled
 by a divide-and-conquer approach:
 in \Cref{sec:cutCells} we cut $\Omega$ by a Cartesian grid
 to generate a set $\mathsf{C}_{\Omega}$ of cut cells. 
In \Cref{sec:mergeCells},
 we merge adjacent cut cells
 so that, for a user-specified value $\epsilon\in(0,1)$,
 the volume fraction of each cut cell is no less than $\epsilon$,
 thus preventing small volumes of the cut cells
 to ensure good conditioning of spatial discretizations
 in \Cref{sec:discretization}. 
 
\subsection{Cut cells}
\label{sec:cutCells}

We embed the domain $\Omega$
 inside an open rectangle $\Omega_R\supset \Omega$ 
 and divide $\Omega_R$ by a Cartesian grid 
 into square control volumes or \emph{cell}s, 
\begin{equation}
  \label{eq:controlVolume}
  \mathbf{C}_\bmi \coloneqq 
  \left(\bmi h, (\bmi+\mathbbm{1})h \right), 
\end{equation}
where $h$ is the uniform grid size, 
$\bmi\in\mathbb{Z}^2$ a multi-index, 
and $\mathbbm{1}\in\mathbb{Z}^2$ the multi-index 
with all components equal to one.
Note that the assumption of $h$ being uniform
is for ease of exposition only, 
as our algorithm also applies to non-uniform grids.

We initialize the $\bmi$th \emph{cut cell} as 
$\calC_\bmi = \mathbf{C}_\bmi \cap \Omega$
and call $\calC_\bmi$ an \emph{empty cell}
if $\mathcal{C}_\ibold = \emptyset$,
a \emph{regular cell} if $\mathcal{C}_\ibold = \mathbf{C}_\ibold$, 
or an \emph{irregular cell} otherwise.
Along the $d$th dimension,
 the \emph{higher face} and the \emph{lower face}
 of the $\bmi$th cell $\mathbf{C}_{\bmi}$ 
 are respectively given by
\begin{equation}
  \label{eq:higherAndLowerFaces}
  \rmF_{\bmihalf} := \left( (\bmi + \basis^d) h,
    (\bmi + \bbone) h \right), \quad
  \rmF_{\ibold-\frac{1}{2}\basis^d}
  := \left( \bmi h,
    (\bmi + \bbone - \basis^d) h \right), 
\end{equation}
where $\basis^d \in \bbZ^2$ is the multi-index 
whose components are all zero except the $d$th component being one.
The \emph{higher/lower cut faces} 
 and the \emph{cut boundary} of the $\bmi$th cell $\mathbf{C}_{\bmi}$ 
 are respectively given by
\begin{equation}
  \label{eq:cutFaceAndCutBoundary}
  \calF_{\ibold\pm\frac{1}{2}\basis^d}
  := \rmF_{\ibold\pm\frac{1}{2}\basis^d} \cap \Omega
  \quad\text{ and }\quad
  \calB_{\bmi} := \mathbf{C}_{\bmi} \cap \partial \Omega. 
\end{equation}


For a domain $\Omega$ and its embedding rectangle $\Omega_R$, 
 the \emph{set of cut cells} is defined as
\begin{equation}
  \label{eq:setOfCutCells}
  \mathsf{C}_{\Omega} := \left\{
    \mathcal{C}_\ibold: 
    \mathcal{C}_\ibold\ne \emptyset; \ 
    \cup_{\ibold}^{\bot\bot} \mathcal{C}_\ibold = \Omega
  \right\}.
\end{equation}
Thanks to the rectangular structure of the Cartesian grid,
the \emph{set of neighbors of a cut cell} $\mathcal{C}_\jbold$
is easily obtained as 
 $ \mathsf{N}_{\jbold} := \left\{
    \mathcal{C}_\ibold:
    \mathcal{C}_\ibold\in \mathsf{C}_{\Omega},\ 
    \|\ibold-\jbold\|_1=1  \right\}$.
%
The connected components of $\mathcal{C}_{\mathbf{j}}$
 are denoted by $\mathcal{C}_{\mathbf{j}}^1$,
 $\mathcal{C}_{\mathbf{j}}^2$, $\ldots$  
 so that $\mathcal{C}_{\mathbf{j}}=
 \cup^{\perp\perp}_k \mathcal{C}_{\mathbf{j}}^k$.
The \emph{set of neighboring components of
 a cut-cell component} $\mathcal{C}_{\mathbf{j}}^k$
 is defined as
\begin{equation}
  \label{eq:neighborCutCellComponents}
  \mathsf{N}_{\jbold}^k := \left\{
    \mathcal{C}_\ibold^{\ell}:
    \mathcal{C}_\ibold\in \mathsf{C}_{\Omega},\ 
    \|\ibold-\jbold\|_1=1,\
    \overline{\mathcal{C}_{\mathbf{j}}^{k}}
    \cap
    \overline{\mathcal{C}_{\mathbf{i}}^{\ell}} \neq
    \emptyset \right\}.
\end{equation}
Note that a neighboring cut cell $\mathcal{C}_\ibold$
 might have multiple components in $\mathsf{N}_{\jbold}^k$.

\subsection{Resolving the small-volume problem
  by cell merging}
\label{sec:mergeCells}

In practice, the \emph{volume} $\lVert \calC_\bmi \rVert$ 
of a cut cell $\calC_\bmi$ may be very small,
leading to ill-conditioning of the discretized operator.
This problem has been addressed by a number of approaches such as
cell merging~\cite{Ji2010},
redistribution~\cite{Almgren1994}, 
and special discretization schemes~\cite{Forrer1998}. 

Our solution of the small-volume problem
 is a novel cell-merging algorithm
 whose output is $\mathsf{C}_{\epsilon}^h(\Omega)$,
 a regularized set of cut cells of $\Omega$
 where each cut cell has only one connected component
 and the volume fraction of this component
 is no less than a user-specified lower bound $\epsilon$: 
\begin{equation}
  \label{eq:mergedCutCellSet}
  \mathsf{C}_{\epsilon}^h(\Omega) :=\left\{
    \mathcal{C}_\ibold: 
    \mathcal{C}_\ibold=\mathcal{C}_\ibold^1 \ne \emptyset; \ 
    \cup_{\ibold}^{\bot\bot} \mathcal{C}_\ibold = \Omega;\ 
    \|{\cal C}_{\bmi}\|\ge\epsilon h^2
  \right\},
\end{equation}
where $\mathcal{C}_\ibold=\mathcal{C}_\ibold^1$ means
 that each $\mathcal{C}_{\ibold}$
 consists of only one connected component.
As a design choice to ensure
 that the unknown function $u$ in (\ref{eq:ccEllipticEq})
 has only one cell average per cut cell,
 the condition $\mathcal{C}_\ibold=\mathcal{C}_\ibold^1$ 
 retains the simplicity of rectangular grids
 and facilitates the design of the multigrid solver in \Cref{sec:multigrid}.

\begin{algorithm}
  \caption{CutAndMergeCells}
  \label{alg:merging}
    \textbf{Input:} $\Omega\in\mathbb{Y}_c$: the problem domain; 
    \\ $\parbox[t]{12.5mm}\  \Omega_R$:
    the rectangle that contains $\Omega$; 
    \\ $\parbox[t]{12.5mm}\  h$:
    the size of the Cartesian grid that discretizes $\Omega_R$;
    \\ $\parbox[t]{12.5mm}\  \epsilon\in(0,1)$: a user-specified threshold
    of small volume fractions.\\
    \textbf{Precondition:} $h$ is sufficiently small
    to resolve the topology and geometry of $\Omega$.
    \\
    \textbf{Output:}  $\mathsf{C}_{\epsilon}^h$:
    the set of cut cells of $\Omega$ in (\ref{eq:mergedCutCellSet}). 
  \begin{algorithmic}[1]
    \State
    $\mathsf{C}_{\epsilon}^h \leftarrow \mathsf{C}_{\Omega} $
    in (\ref{eq:setOfCutCells})
    \For{\textbf{each} cut cell $\mathcal{C}_\ibold = \left(
        \cup^{\bot\bot} \mathcal{C}_\ibold^k \right)_{k=1}^{n_{\ibold}}$
      with $n_{\ibold} \geq 2$ connected components}
    \State
    $\mathcal{C}_\ibold \leftarrow \mathcal{C}_\ibold^m$
    where $m = \argmax_{k=1}^{n_{\ibold}}\{
    \Vert\mathcal{C}_\ibold^k\Vert\}$
    \For{\textbf{each} $k=1,\ldots,m-1,m+1,\ldots,n_{\ibold}$}
    \State
    $\mathcal{C}_\jbold^{i} \leftarrow
    \mathcal{C}_\jbold^{i} \cup^{\bot\bot} \mathcal{C}_\ibold^k$
    where $\mathcal{C}_\jbold^{i} =
    \arg\min_{{\cal C}_{\mathbf{j}'}^{\ell}\in \mathsf{N}_{\ibold}^k}
    \left| \Vert\mathcal{C}_{\mathbf{j}'}^{\ell} \Vert +
    \Vert\mathcal{C}_{\mathbf{i}}^k \Vert  - h^2\right|$
    \EndFor
    \EndFor
    \For{\textbf{each} cut cell $\mathcal{C}_\ibold$
      satisfying $\Vert \mathcal{C}_\ibold \Vert < \epsilon h^2$}
    \State
    $\mathcal{C}_\jbold \leftarrow
    \mathcal{C}_\jbold \cup^{\bot\bot} \mathcal{C}_\ibold$
    where $\mathcal{C}_\jbold =
    \arg\min_{{\cal C}_{\mathbf{j}'}\in \mathsf{N}_{\ibold}^1}
    \left| \Vert\mathcal{C}_{\mathbf{j}'} \Vert +
    \Vert\mathcal{C}_{\mathbf{i}} \Vert  - h^2\right|$
    \State
    $\mathsf{C}_{\epsilon}^h \leftarrow
    \mathsf{C}_{\epsilon}^h \setminus \mathcal{C}_\ibold$
    \EndFor
    \State \Return $\mathsf{C}_{\epsilon}^h$ 
  \end{algorithmic}	
\end{algorithm}

\begin{figure}
    \centering
  \begin{tikzpicture}[scale = 1]
    \fill[gray!30] (-2,-2) rectangle (2.5,2);

    \def\r{1.13}
    \def\eps{0.5}
    \def\w{7}

    \path[save path=\irregularshape]
    plot[domain=0:360,samples=100,smooth cycle] 
    ({(\r+\eps*sin(\w*\x))*cos(\x+3)+0.2},{(\r+\eps*sin(\w*\x))*sin(\x+3)-0.05});

    \draw[step=0.5cm,gray,very thin] (-2,-2) grid (2.5,2);

    \fill[white,use path=\irregularshape];
    
     \begin{scope}
        \clip[use path=\irregularshape];
        \draw[step=0.5cm,gray,very thin,dashed] (-2,-2) grid (2.5,2);
    \end{scope}

    \draw[line width=0.2mm, gray!30] (0,-1.1) -- (0, -1.495);
    \draw[line width=0.2mm, gray!30] (0.5,-1.3) -- (0.5, -1.495);

    \draw[line width=0.2mm, gray!30] (1.005,-1.0) -- (1.15, -1.0);
    \draw[line width=0.2mm, gray!30] (1.5,-0.505) -- (1.5, -0.775);

    \draw[line width=0.2mm, gray!30] (1.5,0.008) -- (1.5, 0.1);
    \draw[line width=0.2mm, gray!30] (1,-0.3) -- (1, -0.1);

    \draw[line width=0.2mm, gray!30] (0.7,0.5) -- (0.995, 0.5);
    \draw[line width=0.2mm, gray!30] (1.005,0.5) -- (1.2, 0.5);

    \draw[line width=0.2mm, gray!30] (-0.005,1) -- (-0.1, 1);
    
    \draw[line width=0.2mm, gray!30] (-0.5,-0.33) -- (-0.5, -0.18);
    \draw[line width=0.2mm, gray!30] (-0.5,0.335) -- (-0.5, 0.495);
    \draw[line width=0.2mm, gray!30] (-0.5,0.505) -- (-0.5, 0.7);
    
    \draw[<->] (0.1,-1.375) -- (-0.25,-1.375);
    \draw[<->] (0.4,-1.375) -- (0.75,-1.375);
    
    \draw[<->] (0.4,-0.875) -- (0.75,-0.875);

    \draw[<->] (-0.05,0.75) -- (-0.05,1.25);
    \draw[<->] (0.875,0.75) -- (0.875,0.375);
    \draw[<->] (1.75,0.125) -- (1.375,0.125);
    \draw[<->] (1.15,0.75) -- (1.15,0.375);
    \draw[<->] (1.25,-0.25) -- (0.75,-0.25);
    \draw[<->] (1.375,-0.625) -- (1.75,-0.625);
    \draw[<->] (1.125,-0.85) -- (1.125,-1.25);

    \draw[<->] (-0.75,0.4) -- (-0.375,0.4);
    \draw[<->] (-0.75,0.65) -- (-0.375,0.65);
    \draw[<->] (-0.25,-0.25) -- (-0.75,-0.25);
        
        \node at (-0.25,-1.25) {\small $\mathbf{i}$};
        \node at (0.75,-1.25) {\small $\mathbf{j}$};
        \node at (0.25,-1.25) {\small $\mathbf{k}$};
        \node at (-1.75,1.75) {\small $\mathbf{l}$};
        \node at (0.25,0.25) {\small $\mathbf{p}$};
        \node at (1.85,0.25) {\small $\mathbf{q}$};

        \node at (2.25,-1.75) {\(\Omega \)};

    \draw[black,use path=\irregularshape];
\end{tikzpicture}
\caption{An illustration of \Cref{alg:merging}
  in generating $\mathsf{C}_{\epsilon}^h(\Omega)$ in (\ref{eq:mergedCutCellSet})
  with $\epsilon=0.2$ 
  by cutting and merging cells for the domain $\Omega$.
  The cut cells ${\cal C}_{\mathbf{l}}$,
   ${\cal C}_{\mathbf{p}}$, and ${\cal C}_{\mathbf{q}}$ are 
   regular, empty, and irregular, respectively.
  The symbol ``$\leftrightarrow$'' indicates cell merging.
  Originally, ${\cal C}_{\mathbf{k}}\in \mathsf{C}_{\Omega}$ is 
   an irregular cell with two small connected components, 
   which are merged to $\mathcal{C}_{\mathbf{j}}$ at line 5
   and $\mathcal{C}_{\mathbf{i}}$ at line 9, respectively.
   Then ${\cal C}_{\mathbf{k}}$ is removed from
   $\mathsf{C}_{\epsilon}^h(\Omega)$
   and its type changed from ``irregular'' to ``empty.''
   The type of ${\cal C}_{\mathbf{i}}$ is changed
   from ``regular'' to ``irregular''
   whereas the type of ${\cal C}_{\mathbf{j}}$
   remains ``irregular.''
 }
  \label{fig:gridsExample}
\end{figure}
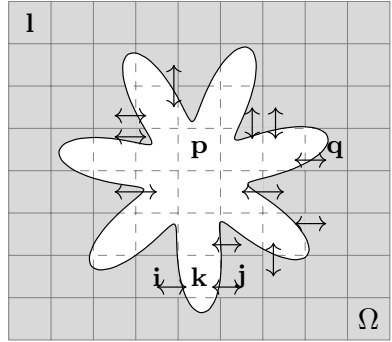

In \Cref{alg:merging},
 we first initialize $\mathsf{C}_{\epsilon}^h$
 with $\mathsf{C}_{\Omega}$ in (\ref{eq:setOfCutCells}) in line 1.
For each cut cell ${\cal C}_{\mathbf{i}}\in \mathsf{C}_{\epsilon}^h$
 with multiple components,
 we identify ${\cal C}_{\mathbf{i}}^m$ as
 a component of ${\cal C}_{\mathbf{i}}$ with the maximum volume, 
 set ${\cal C}_{\mathbf{i}}={\cal C}_{\mathbf{i}}^m$, 
 and merge any other component to a neighboring component
 ${\cal C}_{\mathbf{j}}^{i}$
 so that the volume fraction of the merged component
 is closest to 1; see line 5 and (\ref{eq:neighborCutCellComponents}). 
When the grid size $h$ is too large, 
 the geometry of the boundary may not be well resolved
 so that at the exit of the double loop in lines 2--7
 there may still exist multiple cut cells associated with a single
 control volume.
However,
 so long as $h$ is sufficiently small,
 the loop in lines 2--7 would leave each cut cell with only one component.
Finally in lines 8--11, 
 each small cut cell ${\cal C}_{\ibold}$
 is merged to a neighboring cell
 and removed from the set of cut cells.
 
An example of \Cref{alg:merging} is shown in \Cref{fig:gridsExample}.
After line 1,
 the cut cell $\mathcal{C}_{\mathbf{k}}\in \mathsf{C}_{\Omega}$
 has two connected components: 
 $\mathcal{C}_{\mathbf{k}}^1$ on the left has a larger volume
 than $\mathcal{C}_{\mathbf{k}}^2$ on the right. 
During the double loop in lines 2--7,
 $\mathcal{C}_{\mathbf{k}}$ is set to $\mathcal{C}_{\mathbf{k}}^1$ 
 and $\mathcal{C}_{\mathbf{k}}^2$ is merged with 
 $\mathcal{C}_{\mathbf{j}}^1$. 
During the loop in lines 8--11,
 all cut cells with small volume fractions are merged
 with a neighboring cut cell;
 in particular, 
 $\mathcal{C}_{\mathbf{k}}$ is merged with $\mathcal{C}_{\mathbf{i}}$.


\section{Discretizing equation (\ref{eq:ccEllipticEq})
  into a linear system of cell averages}
\label{sec:discretization}

The \emph{cell average of a scalar function}
$\phi: \overline{\Omega}\rightarrow \mathbb{R}$
over a cut cell $\calC_{\ibold}\in \mathsf{C}_{\epsilon}^h(\Omega)$
is defined as
\begin{equation}
  \label{eq:cellAvg}
  \brk{\phi}_{\bmi} :=
  \frac {1} {\lVert \calC_\bmi \rVert}
  \int_{\calC_{\bmi}} \phi(\bmx) \ \rmd \bmx,
\end{equation}
where $\lVert \calC_\bmi \rVert$ is the \emph{volume} of $\calC_\bmi$;
similarly, the \emph{face average over the cut face} $\calF_{\bmihalf}$
and the \emph{boundary average over the cut boundary}
$\mathcal{B}_\ibold$
are respectively 
\begin{equation}
  \label{eq:faceAvgs}
  \brk{\phi}_{\bmihalf} :=
  \frac {1} {\lVert \calF_{\bmihalf} \rVert}
  \int_{\calF_{\bmihalf}} \phi(\bmx) \ \rmd \bmx\quad
  \text{ and }\quad
  \bbrk{\phi}_{\bmi} := \frac {1} {\Vert \calB_\bmi \Vert}
  \int_{\calB_\bmi} \phi(\bmx) \ \rmd \bmx,
\end{equation}
where $\|\cdot\|$ denotes the length of a cut face or cut boundary
in (\ref{eq:cutFaceAndCutBoundary}). 

The goal of this section is
to \emph{discretize integrals of the elliptic equation
	(\ref{eq:ccEllipticEq}) into a linear system}, 
where the unknowns are the cell averages $\brk{u}_{\ibold}$ 
over the cut cells in (\ref{eq:mergedCutCellSet}).
In \Cref{sec:symmetricFD-cells}, 
we discretize the operator ${\cal L}$ in (\ref{eq:ccEllipticEq})
on SFV cells where symmetric FV formulas apply.
The discretization of (\ref{eq:ccEllipticEq}) on PLG cells
are elaborated in \Cref{sec:PLGCells}.
The final form of the linear system
is summed up in \Cref{sec:discreteElliptic}.

\subsection{Discretizing (\ref{eq:ccEllipticEq}) on SFV cells}
\label{sec:symmetricFD-cells}

A face $\rmF_{\ibold+\frac{1}{2}\basis^d}$
or $\rmF_{\ibold-\frac{1}{2}\basis^d}$
in (\ref{eq:higherAndLowerFaces})
is said to be \emph{extendable}
if it is entirely contained in $\partial \Omega$.
Write
\begin{equation}
  \label{eq:stencilsLop}
  {\cal S}_{\ibold}^{d,+}:=
  \{{\cal C}_{\ibold-m\basis^d}: m=0,1,2,3\},\quad
  {\cal S}_{\ibold}^{d,-}:=
  \{{\cal C}_{\ibold+m\basis^d}: m=0,1,2,3\}.
\end{equation}
A cut cell
$\calC_{\ibold}\in \mathsf{C}_{\epsilon}^h(\Omega)$
is \emph{extendable in the high direction along the $d$th dimension}
if the face
$\rmF_{\ibold+\frac{1}{2}\basis^d}$ is extendable
and all cut cells in ${\cal S}_{\ibold}^{d,+}$
are regular;
similarly,
$\calC_{\ibold}$ 
is \emph{extendable in the low direction along the $d$th dimension}
if the face
$\rmF_{\ibold-\frac{1}{2}\basis^d}$ is extendable
and all cut cells in ${\cal S}_{\ibold}^{d,-}$ are regular.

\begin{figure}
  \centering
  \includegraphics[width = 0.6\linewidth]{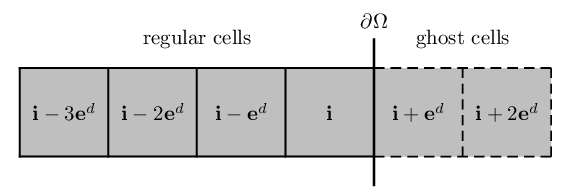}
  \caption{An example of ghost filling near the regular boundary.
    $\rmF_{\ibold+\frac{1}{2}\basis^d}$ is an extendable face
    and ${\cal C}_{\ibold}$ is an extendable cell
    in the high direction along the first dimension.
  }
  \label{fig:GhostCells}
\end{figure}

For an extendable cell,
we append two (regular) ghost cells to each extendable face
in the corresponding direction
and smoothly extend cell averages
of $\phi$
to the ghost cells.
For the example in \Cref{fig:GhostCells},
we use 
\begin{equation}
  \fontsize{7.5}{10}
  \begin{array}{l}
    \brk{\phi}_{\ibold + \mathbf{e}^d}
     = \frac{1}{12} \left(
      3 \brk{\phi}_{\ibold - 3\mathbf{e}^d}
      -17 \brk{\phi}_{\ibold - 2\mathbf{e}^d}
      + 43 \brk{\phi}_{\ibold - \mathbf{e}^d}
      -77 \brk{\phi}_{\ibold}
      +60 \brk{\phi}_{\ibold+\frac{1}{2}\mathbf{e}^d}\right)
      + O(h^5);                \\
    \brk{\phi}_{\ibold + 2 \mathbf{e}^d}
     = \frac{1}{12} \left(
      27 \brk{\phi}_{\ibold - 3\mathbf{e}^d}
      -145 \brk{\phi}_{\ibold - 2\mathbf{e}^d}
      +335 \brk{\phi}_{\ibold - \mathbf{e}^d}
      -505 \brk{\phi}_{\ibold}
      +75 \brk{\phi}_{\ibold+\frac{1}{2}\mathbf{e}^d}
      \right)+ O(h^5)
  \end{array}
\end{equation}
to fill ghost cells while enforcing the Dirichlet boundary condition
$\brk{\phi}_{\ibold+\frac{1}{2}\mathbf{e}^d}$.
As for the Neumann boundary condition
$\brk{\frac {\partial \phi} {\partial
			\mathbf{n}} }_{\ibold+\frac{1}{2}\mathbf{e}^d}$,
the ghost-filling formulas are
\begin{equation}
  \fontsize{7.5}{10}
  \begin{array}{l}
    \brk{\phi}_{\ibold + \mathbf{e}^d}
     = \frac{1}{10} \left(
      \brk{\phi}_{\ibold - 3\mathbf{e}^d}
      -5 \brk{\phi}_{\ibold - 2\mathbf{e}^d}
      +9 \brk{\phi}_{\ibold - \mathbf{e}^d}
      +5 \brk{\phi}_{\ibold}
      +12 h \brk{\frac{\partial \phi}{\partial
      \mathbf{n}}}_{\ibold+\frac{1}{2}\mathbf{e}^d}
      \right) + O(h^5),        \\
    \brk{\phi}_{\ibold + 2\mathbf{e}^d}
     = \frac{1}{2} \left(
      3 \brk{\phi}_{\ibold - 3\mathbf{e}^d}
      -15 \brk{\phi}_{\ibold - 2\mathbf{e}^d}
      +29 \brk{\phi}_{\ibold - \mathbf{e}^d}
      -15 \brk{\phi}_{\ibold}
      +12 h \brk{\frac{\partial \phi}{\partial
      \mathbf{n}}}_{\ibold+\frac{1}{2}\mathbf{e}^d}
      \right) + O(h^5).
  \end{array}
\end{equation}
For periodic boundary conditions,
the values of ghost cells are copied
directly from those of the corresponding regular cells
inside the domain $\Omega$.

\begin{definition}
  \label{def:SFVCell}
  Recall ${\cal L} = a \frac{\partial^2}{\partial x^2}
  + b \frac{\partial^2}{\partial x \partial y}
  + c \frac{\partial^2}{\partial y^2}$
  from (\ref{eq:ccEllipticEq}) and write
  \begin{equation}
    \label{eq:stencilOfIbold}
    \mathcal{S}_{\ibold}:=
    \begin{cases}
      \{ {\cal C}_{\jbold}: \jbold=\ibold+m\ebold^d; d = 1, 2; m = 0, \pm 1, \pm 2 \}
      & \text{if } b=0;
      \\
      \{ {\cal C}_{\jbold}: \jbold=\ibold + m_1\ebold^1 + m_2\ebold^2;
      m_1, m_2 = 0, \pm 1, \pm2\}
      & \text{if } b\ne 0.
    \end{cases}
  \end{equation}
  A cut cell $\calC_{\ibold}\in \mathsf{C}_{\epsilon}^h(\Omega)$
  is called an \emph{SFV cell}
  if each cut cell in $\mathcal{S}_{\ibold}$
  is either a regular cell or a ghost cell;
  otherwise it is called a \emph{PLG cell}.
\end{definition}

The case $b=0$ in (\ref{eq:stencilOfIbold}) follows
directly from (\ref{eq:stencilsLop})
and the following symmetric finite-volume discretization
of the first and second derivatives:
\begin{subequations}
  \fontsize{8.5}{10}
  \begin{align}
    \label{eq:stdGrad}
    &\brk{\frac {\partial \phi} {\partial x_d}}_{\bmi}
     = \frac {1} {12 h} \Big(
      \brk{\phi}_{\bmi - 2 \basis^d}
      - 8 \brk{\phi}_{\bmi - \basis^d}
      + 8 \brk{\phi}_{\bmi + \basis^d}
      - \brk{\phi}_{\bmi + 2 \basis^d}
      \Big)
      + \bigO(h^4),
    \\
    \label{eq:stdLap}
    &\brk{\frac {\partial^2 \phi} {\partial x_d^2}}_{\bmi}
      = \frac {1} {12h^2} \Big(
      -\brk{\phi}_{\bmi - 2\basis^d}
      + 16 \brk{\phi}_{\bmi - \basis^d}
      -30 \brk{\phi}_{\bmi}
      + 16 \brk{\phi}_{\bmi + \basis^d}
      -\brk{\phi}_{\bmi + 2\basis^d}
      \Big) + \bigO(h^4);
  \end{align}
\end{subequations}
see \cite{zhang12} for a proof of the fourth-order accuracy.
The case $b\ne 0$ in (\ref{eq:stencilOfIbold}) follows
from the discretization of the cross derivative
$\frac {\partial^2 } {\partial x_i \partial x_j} (i \neq j)$
by applying (\ref{eq:stdGrad})
first in the $i$th direction and then in the $j$th direction.



\subsection{Discretizing (\ref{eq:ccEllipticEq}) on PLG cells}
\label{sec:PLGCells}

In \Cref{sec:PLG}, 
 we briefly review the PLG algorithm \cite{Zhang:PLG}
 that generates a suitable stencil for each PLG cell.
In \Cref{sec:Discretization-PLG-cells},
 we fit a multivariate polynomial locally
 to discretize $\avg{{\cal L}u}_{\mathbf{i}}$
 as a linear combination of cell averages and boundary averages.
 
\subsubsection{Poised lattice generation (PLG)}
\label{sec:PLG}


We start with
\begin{definition}[Lagrange interpolation problem (LIP)]
  \label{def:LIP}
  Denote by $\Pi_n^{\Dim}$ the vector space of
  all $\Dim$-variate polynomials of degree no more than $n$
  with real coefficients.
  Given points $\bmx_1, \bmx_2, \cdots, \bmx_N \in \bbR^\Dim$
  and the same number of data $f_1, f_2, \cdots, f_N \in \bbR$,
  the \emph{Lagrange interpolation problem} seeks a polynomial $f \in \Pi_n^\Dim$ such that
  \begin{equation}
    \forall j=1, 2, \cdots, N,\quad
    f(\bmx_j) = f_j, 
    \label{eq:finiteDiffLIP}
  \end{equation}
  where $\Pi_n^\Dim$ is the \emph{interpolation space}
  and $\{\bmx_j\}_{j=1}^N$ are the \emph{interpolation sites}.
  \label{def:finiteDiffLIP}
\end{definition}

The sites $\left\{ \bmx_j \right\}_{j=1}^N$
 are \emph{poised} in $\Pi_n^\Dim$ if,
 for any given data $\left\{ f_j \right\}_{j=1}^N$,
 there exists a unique $f \in \Pi_n^\Dim$ satisfying (\ref{eq:finiteDiffLIP}). 
For a basis $\left\{ \phi_j \right\}$ of $\Pi_n^\Dim$, 
 $\left\{ \bmx_j \right\}_{j=1}^N$ are poised
 if and only if $N = \dim \Pi_n^\Dim = \binom{n+\Dim}{n}$
 and the following \emph{sample matrix} $M\in \mathbb{R}^{N\times N}$
 is non-singular, 
 \begin{equation}
   \label{eq:sampleMatrix}
   \forall j, k = 1,2, \ldots N, \quad
   M(\left\{ \phi_j \right\} ; \left\{ \bmx_k \right\})
   = \Big[ M_{jk} \Big] := \Big[ \phi_j(\bmx_k) \Big].
 \end{equation}

For $\Dim=1$,
 the LIP is unisolvent if and only if its sites are pairwise distinct.
For $\Dim>1$, however, 
 it is difficult to determine 
 whether a set of sites is poised in $\Pi^{\Dim}_n$.
For example,
 the lattice
 $\{(5,0), (-5,0), (0,5), (0,-5), (4,3), (-3,4)\}$
 is not poised in $\Pi^2_2=\Span(1,x,y,x^2,y^2,xy)$
 because the corresponding sample matrix in (\ref{eq:sampleMatrix}) is singular. 
As the core difficulty of multivariate interpolation, 
\emph{the poisedness of interpolation sites in multiple dimensions
  depends on their geometric configuration}.
 
\begin{definition}[PLG in $\mathbb{Z}^\Dim$~\cite{Zhang:PLG}]
  \label{def:PLG}
  Given a finite set  $K\subset \mathbb{Z}^\text{D}$ of feasible nodes,
  a starting point $\mathbf{q}\in K$, and a degree $n\in \mathbb{Z}^+$,
  the \emph{problem of poised lattice generation} is to
  choose a lattice $\mathcal{T}\subset K$ such that
  $\mathcal{T}$ is poised in $\Pi_n^\Dim$ and $\#\mathcal{T} = \dim \Pi_n^\Dim$.
\end{definition}

In \Cref{def:PLG},
 $\mathbb{Z}^{\Dim}$ captures the rectangular structure of FD grids 
 while $K$ reflects the physics
 of the spatial operator being discretized.
For example,
 to discretize an advection operator,
 we set $K$ to be a lopsided box with respect to $\mathbf{q}$
 because most information comes from the upwind direction;
 see \cite[Fig. 5]{Zhang:PLG}.
Considering the isotropy of diffusion
 for the elliptic operator ${\cal L}$ in (\ref{eq:ccEllipticEq}), 
 we set $K$ in this work to be a box centered at $\mathbf{q}$
 as much as possible.
 
Via a fusion of approximation theory,
 group theory, and search algorithms in artificial intelligence, 
 we solved the PLG problem in \Cref{def:PLG}
 by a novel and efficient PLG algorithm \cite{Zhang:PLG},
 which is applied in this work
 to discretize $\avg{{\cal L}u}_{\mathbf{i}}$ with $K$. 

\subsubsection{Approximating $\avg{{\cal L}u}_{\mathbf{i}}$
  with a linear combination of cell averages (and a boundary average)}
\label{sec:Discretization-PLG-cells}

Let $\calS_{\mathrm{PLG}}(\ibold) =
\left\{\calC_{\bmj_1}, \calC_{\bmj_2}, \cdots, \calC_{\bmj_N}\right\}$
denote the poised lattice generated by the PLG algorithm
where $N=\dim{\Pi_n^\Dim}$. 
As shown in \Cref{fig:exampleStencil},
 the \emph{stencil for discretizing $\calL$ over a PLG cell} $\calC_\bmi$ is 
\begin{equation}
   \label{eq:PLGstencil}
   \calX(\bmi) =
   \begin{cases}
     \calS_{\mathrm{PLG}}(\ibold) 
     & \text{if ${\cal C}_{\mathbf{i}}$ is a regular PLG cell}; 
     \\
     \calS_{\mathrm{PLG}}(\ibold) \cup
     \left\{\mathcal{B}_\bmi \right\}
     & \text{if ${\cal C}_{\mathbf{i}}$ is an irregular PLG cell}.
   \end{cases}
\end{equation}

\begin{figure}
  \centering
  \begin{tikzpicture}[scale=0.55]
	\fill[gray!30] (0,0) -- (0,0.2) -- 
	(0,0.2) arc[start angle=290, end angle=350.4,
				radius=8.8cm] -- (9,7) -- (9,0) -- cycle;

        \fill [gray!90]
        (3,2) -- (4,2) -- (4,1) -- (6,1) -- (6,2) -- (8,2) -- (8,4) -- (7,4)
        -- (7,5) -- (6,5) -- (6,6) -- (3,6) -- cycle;

	\foreach \x in {0,...,8}
	\foreach \y in {0,...,6}
		\draw[thin,color = black!60] (\x,\y) rectangle (\x+1,\y+1);
        \draw[white,line width = 0.5mm] (0,0.21)--(0,7)--(5.65,7);
        
	\draw[gray!30,  line width=0.5mm] (1.985, 1.0) -- (1.65, 1.0);
	\draw[gray!90,  line width=0.5mm] (4.0, 3.02) -- (4.0, 3.1);
        
        \fill[white]
		(0,0.2) -- (0,0.2) arc[start angle=290, end angle=350.4,
				radius=8.8cm] -- (0,7) -- cycle;
        \draw[thin, color =black!60] (0,0.2) 
	arc[start angle=290, end angle=350.4,radius=8.8cm];
	\draw[line width = 0.35mm] (3.9,3) 
	arc[start angle=322, end angle=330.2,radius=8.8cm];

        \fill [black!60] (4.5,3.5) circle (3pt);
  \end{tikzpicture}
  \caption{An example of the stencil for multivariate polynomial fitting
    in the FV formulation for $\Dim = 2$ and $n=4$.
    ``$\bullet$'' marks the starting point $\mathbf{q}=\mathbf{i}$, 
    the cells with dark shades constitute
    $\calS_{\mathrm{PLG}}$ in (\ref{eq:PLGstencil}), 
    and the thick curve segment represents
    the cut boundary $\mathcal{B}_\bmi$.
    }
  \label{fig:exampleStencil}
\end{figure}
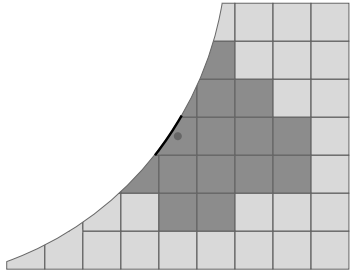



Given the cell averages and the boundary average
\begin{equation}
  \label{eq:cellAvgAndBdryAvg}
  \overline{\mathbf{u}} = \left[ \brk{u}_{\bmj_1}, \cdots , \brk{u}_{\bmj_N},
    \bbrk{\calN u}_{\bmi} \right]^{\top} \in \bbR^{N+1},
\end{equation}
the goal is to determine a vector of coefficients
$\boldsymbol{\beta} = \left[ \beta_1, \cdots, \beta_N, \beta_{\mathrm{b}}\right]^{\top}$
such that the linear combination $\boldsymbol{\beta}^{\top}
\overline{\mathbf{u}}$
is an $(n-1)$th-order approximation of $\avg{{\cal L} u}_{\bmi}$,
\begin{equation}
  \label{eq:goalOfApproximation}
  \forall u \in {\cal C}^{n+1}(\bbR^\Dim, \bbR), \quad 
  \boldsymbol{\beta}^{\top} \overline{\mathbf{u}} = \brk{\calL u}_\bmi
  + \bigO(h^{n-1}), 
\end{equation}
where $u$ can be approximated 
to the $(n+1)$th-order accuracy 
by a complete $\Dim$-variate polynomial with total degree $n$. 
Then $\bigO(h^{n-1})$ follows
from second derivatives in ${\cal L}$.

The equations on $\boldsymbol{\beta}$ are obtained 
 via a restricted version of
 (\ref{eq:goalOfApproximation}), 
\begin{equation}
  \forall u \in \Pi_n^\Dim, \quad
  \brk{\calL u}_{\bmi} = \sum\nolimits_{k=1}^N \beta_k \avg{u}_{\mathbf{j}_k}
  + \beta_{\mathrm{b}} \bbrk{\calN u}_{\bmi}, 
\end{equation}
which is equivalent to
$\brk{\calL \phi_j}_{\bmi} = \sum\nolimits_{k=1}^N \beta_k \brk{\phi_j}_{\bmj_k}
  + \beta_{\mathrm{b}} \bbrk{\mathcal{N}\phi_j}_{\bmi}$
for a basis $\left\{ \phi_j \right\}_{j=1}^N$ of $\Pi_n^\Dim$.
These equations form a linear system
\begin{equation}
  \label{eq:linearSystemOfCoefs}
  M \boldsymbol{\beta} = \overline{\boldsymbol{\phi}},
\end{equation}
where $\overline{\boldsymbol{\phi}}= \left(
  \brk{\calL \phi_1}_{\bmi},
  \brk{\calL \phi_2}_{\bmi},
  \ldots, \brk{\calL \phi_N}_{\bmi} \right)^{\top}\in \bbR^N$; 
for an irregular PLG cell ${\cal C}_{\mathbf{i}}$, we have 
\begin{equation}
  \label{eq:detailedSampleMatrix}
  M  =
      \begin{bmatrix}
        \brk{\phi_1}_{\jbold_1} & \brk{\phi_1}_{\jbold_2} & \cdots & \brk{\phi_1}_{\jbold_N} & \bbrk{\calN \phi_1}_{\bmi} \\
        \brk{\phi_2}_{\jbold_1} & \brk{\phi_2}_{\jbold_2} & \cdots & \brk{\phi_2}_{\jbold_N} & \bbrk{\calN \phi_2}_{\bmi} \\
        \vdots                  & \vdots                  & \ddots & \vdots                  & \vdots                     \\
        \brk{\phi_N}_{\jbold_1} & \brk{\phi_N}_{\jbold_2} & \cdots & \brk{\phi_N}_{\jbold_N} & \bbrk{\calN \phi_N}_{\bmi} \\
      \end{bmatrix}
  \in \bbR^{N \times (N+1)}.
  %
\end{equation}
For a regular PLG cell ${\cal C}_{\mathbf{i}}$, 
 the last column of $M$ in (\ref{eq:detailedSampleMatrix}) is dropped, 
 so are the last elements of 
 $\boldsymbol{\beta}$ and $\overline{\boldsymbol{\phi}}$.
 
We calculate the integrals on regular cells 
 by six-order recursive Gauss formulas.
In contrast, the integral of a scalar function
 over an irregular cut cell 
 is first converted by Green's theorem
 to another integral on the boundary Jordan curve 
 and then approximated by sixth-order Gauss formulas;
 see \cite{zhang13:_highl_lagran} for more details.
Together with the explicit approximation of the boundary Jordan curve with cubic splines,
 this integral formulation makes our method robust
 in that its fourth-order accuracy holds even at the presence of kinks
 on the domain boundary.
 
If (\ref{eq:linearSystemOfCoefs}) is under-determined,
 we solve a constrained optimization problem
 \begin{equation}
   \label{eq:optimizationProb}
   \min_{\boldsymbol{\beta} \in \bbR^{N+1}} \left\Vert \boldsymbol{\beta} \right\Vert_{W^{-1}}
   \quad
   \mathrm{s.t.} \quad M \boldsymbol{\beta} = \overline{\boldsymbol{\phi}}, 
 \end{equation}
where the square matrix $W$ is symmetric positive definite,
the \emph{$W$-inner product of two vectors} 
is $\langle \mathbf{w}, \mathbf{v} \rangle_W := \mathbf{w}^{\top} W \mathbf{v}$, 
and the \emph{$W$-norm of a vector} is
$\Vert \mathbf{v} \Vert_W := \sqrt{\langle \mathbf{v}, \mathbf{v} \rangle_W}$.
 
Since $M$ has full row rank,
it follows from \Cref{lem:leastSquares} that
(\ref{eq:optimizationProb}) is solved by
\begin{equation}
  \label{eq:derivedCoefAgain}
  \boldsymbol{\beta}_{\min} = W M^{\top} \left( M W M^{\top} \right)^{-1} \overline{\boldsymbol{\phi}}.
\end{equation}

\begin{lemma}
  \label{lem:leastSquares}
  Let $A \in \bbR^{m \times N}$ be a matrix with full column rank
  and $W \in \bbR^{m \times m}$ be a symmetric positive definite matrix.
  For any $\mathbf{v} \in \bbR^N$, the optimization problem of 
  $\min_{\mathbf{x} \in \bbR^m} \left\Vert \mathbf{x} \right\Vert_{W^{-1}}$
  under the constraint $A^{\top} \mathbf{x}=\mathbf{v}$
  admits a unique solution \mbox{$\mathbf{x}_{\min} = W A \left(A^{\top} W A\right)^{-1} \mathbf{v}$}.
\end{lemma}
\begin{proof}
  See~\cite[\S 5.3, \S 5.6, \S 6.1]{Golub}.
\end{proof}


It is reasonable to demand
 that a cut cell closer to ${\cal C}_{\mathbf{i}}$
 has a greater influence upon the linear system
 than a cut cell away from ${\cal C}_{\mathbf{i}}$. 
To this end, 
 we set 
\begin{subequations}
  \label{eq:weightDef}
  \begin{align}
    W^{-1} &= \diag\left( w_1, \cdots, w_N, w_{\mathrm{b}} \right);
    \\
    w_k &= \max \left\{ \Vert \bmj_k - \bmi \Vert_2, w_{\min} \right\}; \quad
          w_{\mathrm{b}} = w_{\min}=0.5, 
  \end{align}
\end{subequations}
where the nonzero value of $w_{\min}$ prevents $\frac{1}{w_k}$ from being too large.

To maintain good conditioning,
 the basis $\left\{ \phi_j \right\}_{j=1}^N$ of $\Pi_n^\Dim$
 is set to
 \begin{equation}
   \begin{array}{l}
     \Phi_n^\Dim(h;\bmp) = \left\{ \left( \frac{\bmx - \bmp}{h}
     \right)^{\bmalpha}:
     \bmalpha \in \{0,1,\ldots, n\}^{\Dim} \text{ and }
     
     \|\bmalpha \|_1 \in [0, n]\right\},
   \end{array}
\end{equation}
where $\bmp \in \bbR^\Dim$ is the center of ${\cal C}_{\mathbf{i}}$.


\subsection{The linear system
	as the discrete elliptic problems}
\label{sec:discreteElliptic}

Given $\Omega$, $\Omega_R$, $h$, and $\epsilon$,
 \Cref{alg:merging} uniquely
 determines the set $\mathsf{C}^h_{\epsilon}$ of cut cells
 in (\ref{eq:mergedCutCellSet}).
For SFV cells,
 the symmetric FV formulas in \Cref{sec:symmetricFD-cells}
 are employed to discretize the integral of (\ref{eq:ccEllipticEq}).
For each PLG cell ${\cal C}_{\mathbf{i}}$,
 the vector $\overline{\boldsymbol{\phi}}$
 and the matrices $M$ and $W$
 in (\ref{eq:linearSystemOfCoefs}) and (\ref{eq:weightDef})
 yield $\boldsymbol{\beta}_{\min}$ in (\ref{eq:derivedCoefAgain}),
 which, together with (\ref{eq:goalOfApproximation}),
 implies that
 $\boldsymbol{\beta}_{\min}^{\top} \overline{\mathbf{u}}$
 is an $(n-1)$th-order approximation of
 the integral $\brk{\calL u}_\bmi$ 
 of (\ref{eq:ccEllipticEq}) over ${\cal C}_{\mathbf{i}}$. 
Combine the two cases
 and we have a linear system of the form
\begin{equation}
  \label{eq:PoissonLinearDiscretization}
  A \mathbf{u} = \mathbf{b}:=\mathbf{f} - N \mathbf{g},
\end{equation}
where $\mathbf{f}$ is a vector of cell averages
of the right-hand side (RHS) function $f$ in (\ref{eq:ccEllipticEq})
and the matrices $A$ and $N$ discretize
the elliptic operator ${\cal L}$
and the boundary operator ${\cal N}$, respectively.
Similar to $\overline{\mathbf{u}}$ in (\ref{eq:cellAvgAndBdryAvg}),
$\mathbf{u}$ and $\mathbf{g}$ are the vector of cell averages
and the vector of boundary averages, respectively.
The structure of $A$
 is better revealed by the following block form
 that is equivalent to (\ref{eq:PoissonLinearDiscretization}), 
\begin{equation}
  \label{eq:PoissonSplitDiscretization}
  \begin{bmatrix}
    A_{11} & A_{12} \\
    A_{21} & A_{22}
  \end{bmatrix}
  \begin{bmatrix}
    \mathbf{u}_1 \\
    \mathbf{u}_2
  \end{bmatrix}
  =
  \begin{bmatrix}
    \mathbf{b}_1 \\
    \mathbf{b}_2
  \end{bmatrix}, 
\end{equation}
where $\mathbf{u}_1$ and $\mathbf{u}_2$
contain cell averages of $u$ over SFV cells and PLG cells, 
respectively. 
The eigenvalues of $A_{11}$ have nonnegative real parts. 
In contrast, each of $A_{12}$, $A_{21}$, and $A_{22}$
 is asymmetric and indefinite;
 all we know about them is their sparsity.

The \emph{error} and the \emph{residual} of
 an approximate solution $\tilde{\mathbf{u}}\approx \mathbf{u}$
 of (\ref{eq:PoissonLinearDiscretization}) 
 are respectively defined as
\begin{equation}
  \label{eq:errorAndResidual}
  \mathbf{e}(\tilde{\mathbf{u}})
  := \mathbf{u} - \tilde{\mathbf{u}},\quad
  \mathbf{r}(\tilde{\mathbf{u}})
  := \mathbf{b} - A\tilde{\mathbf{u}}. 
\end{equation}
Then (\ref{eq:PoissonLinearDiscretization}) can be rewritten
 as the equivalent residual equation 
$A \mathbf{e} = \mathbf{r} = \mathbf{b} - A \tilde{\mathbf{u}}$, 
which is conducive to the design
 and exposition of multigrid methods.




\section{The cut-cell geometric multigrid method}
\label{sec:multigrid}

The PLG discretization results in 
 the indefiniteness of the block matrix $A_{22}$
 in (\ref{eq:PoissonSplitDiscretization}), 
 and thus prohibits a direct application
 of traditional geometric multigrid methods. 
To cope with this difficulty,
 we give a total ordering to the set of PLG cells
 and prove in \Cref{sec:order_L22}
 that the LU factorization of the corresponding subblock $A_{22}$
 has the optimal complexity of $O(h^{-1})$. 
Then 
 we design a fixed-point iteration in \Cref{sec:smoother} 
 as a block smoother of (\ref{eq:PoissonSplitDiscretization})
 and assemble these components
 into a cut-cell V-cycle in \Cref{sec:Vcycles}. 
In \Cref{sec:complexity},
 we assemble these components
 to propose a cut-cell full multigrid (FMG) cycle
 as a new cut-cell geometric multigrid method
 that solves the block linear system (\ref{eq:PoissonSplitDiscretization})
 with the optimal complexity of $O(h^{-2})$.

\subsection{An optimal LU factorization of $A_{22}$
  in (\ref{eq:PoissonSplitDiscretization})}
\label{sec:order_L22}

The \emph{bandwidth of a square matrix $A$}
 is the minimum nonnegative integer $k$
 such that $|i-j|>k$ implies $a_{i,j}=0$.
The bandwidth of $A_{22}$ in (\ref{eq:PoissonSplitDiscretization})
 is greatly affected by the ordering of
 the unknowns in $\mathbf{u}_2$,
 i.e., the ordering of the PLG cells.
By the unique representation of Yin sets in (\ref{eq:uniqueRepOfYinSets}), 
 it suffices to define the ordering for PLG cells
 close to a single Jordan curve $\gamma : [0, 1] \to \mathbb{R}^2$
 where $\gamma(0)=\gamma(1)$
 and $\gamma|_{[0,1)}$ is a continuous injection.

\begin{definition}
  \label{def:linearOrderingOfPLGcells}
  For a Jordan curve $\gamma\subset \partial \Omega$, 
  the \emph{total ordering} on the multi-index set
  $\mathcal{I}_{\mathrm{PLG},\gamma}:=\{\ibold:
  {\cal C}_{\ibold} \text{ is a PLG cell near }\gamma\}$ 
  is given by 
  $\ibold \leq \jbold$ if and only if $s(\ibold)\leq s(\jbold)$,
  where $s(\ibold)$ is the parameter of the point $\gamma(s(\ibold))$ on $\gamma$
  that is closest to the center of the cell $\mathbf{C}_{\ibold}$
  in (\ref{eq:controlVolume}). 
\end{definition}
             
\begin{figure}
  \centering
  \begin{tikzpicture}[scale=0.5]

		\fill[color = gray] (16,3) -- (13,3) -- (13,4) -- (10,4) --
		(10,5) -- (8,5) -- (8,6) -- (7,6) -- (7,8) -- (6,8) -- (6,10)
		-- (5,10) -- (5,11) -- (16,11) -- cycle;

		\fill[yellow]
		(2.4,11) -- (5,11) -- (5,10) -- (6,10) -- (6,8) -- (7,8) -- (7,6)
		-- (8,6) -- (8,5) -- (10,5) -- (10,4) -- (13,4) --
		(13,3) -- (16,3) -- (16,2) -- (9.8,2) --
		(9.8,2)arc[start angle=244, end angle=194.5, radius=14cm]
		-- cycle;


		\foreach \x in {0,...,15}
		\foreach \y in {2,...,10}
		\draw[thin,color = gray!80] (\x,\y) rectangle (\x+1,\y+1);


                \foreach \x in {0,2,4,6,8,10,12,14}
		\foreach \y in {2,4,6,8}
		\draw[thick,dashed, color =black] (\x,\y)
                rectangle (\x+2,\y+2);

                \foreach \x in {0,2,4,6,8,10,12,14}
                \draw[thick,dashed,color = black] (\x,11) -- (\x,10) --
                (\x+2,10) -- (\x+2,11);


		\draw[yellow,  line width=0.6mm] (3.0, 9.99) -- (3.0, 9.25);
		\draw[yellow,  line width=0.5mm] (4.0, 7.995) -- (4.0, 7.35);
		\draw[yellow,  line width=0.5mm] (5.0, 5.99) -- (5.0, 5.9);
		\draw[yellow,  line width=0.5mm] (6.0, 4.995) -- (6.0, 4.7);
		\draw[yellow,  line width=0.5mm] (7.0, 3.99) -- (7.0, 3.8);
                
		\draw[thick] (9.8,2) arc[start angle=244, end angle=194.5,
				radius=14cm];

		\node at (2.6,10.5) {1};     \node at (3.5,10.5) {2};
                \node at (3.5,9.5) {4};       \node at (4.5,10.5) {3};
                \node at (3.5,8.5) {6};       \node at (4.5,9.5) {5};    
                \node at (4.5,8.5) {8};       \node at (5.5,9.5) {7};
                \node at (4.5,7.5) {10};
                \node at (5.5,8.5) {9};   
		\node at (4.5,6.5) {11};     \node at (5.5,6.5) {13};
		  \node at (5.5,7.5) {12};
		     \node at (6.5,7.5) {14};
		 
		 \node at (6.5,4.5) {18};

		\node at (5.5,5.5) {15};     \node at (6.5,5.5) {17};

		\node at (6.5,6.5) {16};

		\node at (7.5,3.5) {21};     \node at (7.5,4.5) {20};

		\node at (7.5,5.5) {19};
		\node at (8.5,4.5) {$\ldots$};
		\node at (8.5,2.5) {$\ldots$};

		\node at (12,7) {\large SFV cells};

		\node at (12,2.5) {\large PLG cells};
	\end{tikzpicture}
  \caption{Illustrating the total ordering of PLG cells
    in \Cref{def:linearOrderingOfPLGcells}. 
    The SFV and PLG cells are shaded in gray and in yellow,
    respectively.
    The dashed boxes represent the coarse cells.
    Fine SFV cells
    (such as the two adjacent to \#14 and \#16)
    may be covered by a coarse PLG cell.}
  \label{fig:OrderOnBdry}
\end{figure}
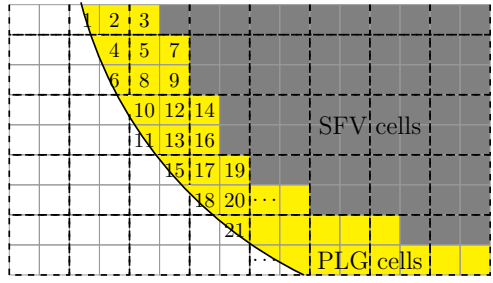

See Figure~\ref{fig:OrderOnBdry} for an illustration
 of \Cref{def:linearOrderingOfPLGcells}. 
When $\partial\Omega$ consists of multiple Jordan curves,
 we order PLG cells near each Jordan curve consecutively. 

\begin{lemma}
  \label{lem:bandWidth}
  Suppose $\partial \Omega$ consists of only one single Jordan curve
  and the grid size $h$ is sufficiently small to resolve $\partial \Omega$.
  Then the total ordering in \Cref{def:linearOrderingOfPLGcells}
  and the PLG discretization in \Cref{sec:PLGCells} with $n=4$
  yield
  \begin{equation}
    \label{eq:A22}
    A_{22} = A_{22,c} +
    \begin{bmatrix}
      \mathbf{0} & A_{22,u} \\
      \mathbf{0} & \mathbf{0}
    \end{bmatrix}
    +
    \begin{bmatrix}
      \mathbf{0}        & \mathbf{0} \\
      A_{22,l} & \mathbf{0}
    \end{bmatrix},
  \end{equation}
  where $\mathbf{0}$ represents a block matrix whose elements are all zero; 
  the bandwidth of $A_{22,c}$ is at most 17,
  so are the dimensions of the square blocks $A_{22,u}$ and $A_{22,l}$.
\end{lemma}
\begin{proof}
  By $n=4$ and \Cref{def:SFVCell},
  an SFV cell is at the center of a 5-by-5 box of regular cells.
  A nonempty cut-cell is either an SFV cell or a PLG cell
  and each row of $A_{22}$ corresponds to a PLG cell.
  Therefore, the distance from the center of a PLG cell
  to $\gamma$ is at most $\frac{5\sqrt{2}}{2}h$.
  In the worst case,
  the set $K$ of feasible nodes in \Cref{def:finiteDiffLIP}
  is a 5-by-5 box, 
  of which the starting point $\mathbf{q}$ is at the box corner. 
  By the total ordering in \Cref{def:linearOrderingOfPLGcells}, 
  the difference between the numbering of $\mathbf{q}$
  and that of any multi-index in the box
  is bounded by $5\times \frac{5\sqrt{2}}{2}\approx 17.7$. 

  The above argument does not hold
  in a local neighborhood of $\gamma(0)$,
  where the difference of the numbering of two such PLG cells
  might be close to the total number of PLG cells.
  However, the number of these pairs of PLG cells is $O(1)$
  and these large differences in PLG cell numbering 
  can be assimilated either into $A_{22,l}$ or into $A_{22,u}$,
  whose dimensions, by similar arguments as above,  are at most 17.
\end{proof}

\begin{definition}
  \label{def:A22_LU}
  The \emph{LU factorization of $A_{22}$ in (\ref{eq:PoissonSplitDiscretization})}
  is given by
  \begin{equation}
    \label{eq:A22_LU}
    A_{22} :=
    \begin{bmatrix}
      B & P \\
      Q & S
    \end{bmatrix}
    = 
    \begin{bmatrix}
      L_B & \mathbf{0} \\
      Y & L_S
    \end{bmatrix}
    \begin{bmatrix}
      U_B & X \\
      \mathbf{0} & U_S
    \end{bmatrix}
    =: L_{22} U_{22}, 
  \end{equation}
  where diag$(B,S)=A_{22,c}$ in (\ref{eq:A22}),
  $P$ and $Q$ correspond to $A_{22,u}$ and $A_{22,l}$
  padded with zeros, respectively, 
  and the other subblocks are obtained by steps as follows.
  \begin{enumerate}[label={(\alph*)}]
  \item \label{line:B_LU}
    Perform an LU factorization on $B\in \mathbb{R}^{m\times m}$ to get
    $B = L_BU_B$;
  \item Solve $L_BX=P$ for $X\in\mathbb{R}^{m\times k}$
    by $k$ forward substitutions; 
  \item Solve $YU_B = Q$ for $Y \in \mathbb{R}^{k\times m}$
    by $k$ backward substitutions; 
  \item Perform an LU factorization on $S' = S-YX \in \mathbb{R}^{k\times k}$
    to get $S' = L_SU_S$.
  \end{enumerate}
\end{definition}

To examine the complexity of the above LU factorization, 
 we need
\begin{lemma}
  \label{lem:BandedLU}
  Suppose $A\in\mathbb{R}^{m\times m}$ has an LU factorization
  $A=LU$ and the bandwidth of $A$ is $p$. 
  Then the bandwidths of $L$ and $U$ are both $p$.
  In addition, the complexity of this factorization
  via Gaussian elimination is $O(m p^2)$.
\end{lemma}
\begin{proof}
  The first conclusion follows from \cite[Theorem 4.3.1]{Golub}.
  In the $k$th step of the Gaussian elimination, 
  all non-zero elements from the $(k+1)$th row to the
  $\min(k+p,m)$th row need to be annihilated.
  Therefore, the total number of floating-point operations
  in this LU factorization is
  $\sum_{k=1}^{m-1}2\min(p,m-k) \cdot \min(p+1, m-k+1)$, 
  yielding a complexity of $O(m p^2)$.
\end{proof}


\begin{theorem}
  \label{thm:ComplexityBandedLUL22}
  For $A_{22}\in \mathbb{R}^{m\times m}$ in
  (\ref{eq:PoissonSplitDiscretization}), 
  we have $m=\bigO(h^{-1})$
  and the complexity of
  the LU factorization of $A_{22}$ in \Cref{def:A22_LU}
  is also $\bigO(h^{-1})$.
\end{theorem}
\begin{proof}
  $m=\bigO(h^{-1})$ follows from
  $\partial \Omega$ being a set of codimension one in $\Omega_R$.
  
  For steps in \Cref{def:A22_LU}, 
  \Cref{lem:bandWidth,lem:BandedLU} imply
  that the complexity of (a) is $\bigO(m)$, 
  \Cref{lem:bandWidth} dictates that
  the complexity of each of (b) and (c) is $\bigO(m)$
  and that the dimension of $S$ is $O(1)$, 
  and thus the complexity of (d) is also $\bigO(1)$. 
\end{proof}




\subsection{A block smoother} 
\label{sec:smoother}

A \emph{fixed point iteration
  for solving a linear system} $A\mathbf{u} = \mathbf{b}$
is an iteration of the form
$\mathbf{u}^{(k+1)} = T\mathbf{u}^{(k)} + \mathbf{c}$
where $\mathbf{u}^{(k)}$ is the $k$th iterate that approximates $\mathbf{u}$
 while $T$ and $\mathbf{c}$ are functions of $A$ and $\mathbf{b}$
 satisfying $\mathbf{u}=T\mathbf{u}+\mathbf{c}$.

The \emph{Jacobi iteration}
 is a fixed point iteration 
 with $T_J = I - D^{-1}A$, $\mathbf{c}_J=D^{-1}\mathbf{b}$,
 where $D$ is the diagonal part of $A$.
The \emph{weighted Jacobi iteration} is another fixed point iteration
of the form
\begin{equation}
  \label{eq:weightedJacobi}
  \mathbf{u}^{(k+1)} :=
  (1-\omega)\mathbf{u}^{(k)} + \omega\mathbf{u}_*
  =(I-\omega D^{-1}A)\mathbf{u}^{(k)} + \omega\mathbf{c}_J, 
\end{equation}
where $\mathbf{u}_* = T_J\mathbf{u}^{(k)} + \mathbf{c}_J$.
Due to the indefiniteness of $A_{22}$,
 a direct application to (\ref{eq:PoissonSplitDiscretization})
 would result in divergence.

Exploiting the block structure of (\ref{eq:PoissonSplitDiscretization})
 and the optimal complexity of the LU factorization in
 \Cref{thm:ComplexityBandedLUL22}, 
 we propose 
 
\begin{definition}
  \label{def:blockSmoother}
  The \emph{block smoother 
    for the linear system (\ref{eq:PoissonSplitDiscretization})}
  is a fixed point iteration of the form  
  \begin{equation}
    \label{eq:blockSmoother}
    \begin{bmatrix}
      \mathbf{u}_1 \\ \mathbf{u}_2
    \end{bmatrix}^{(k+1)}
    = T_{\omega}
    \begin{bmatrix}
      \mathbf{u}_1 \\ \mathbf{u}_2
    \end{bmatrix}^{(k)}
    +
    \begin{bmatrix}
      \omega D_{11}^{-1} & \mathbf{0}
      \\
      -\omega U_{22}^{-1}L_{22}^{-1}A_{21}D_{11}^{-1} & U_{22}^{-1}L_{22}^{-1}
    \end{bmatrix}
    \begin{bmatrix}
      \mathbf{b}_1
      \\
      \mathbf{b}_2
    \end{bmatrix},
  \end{equation}
  where $L_{22}U_{22}=A_{22}$ is the LU factorization in
  (\ref{eq:A22_LU}),
  $D_{11}$ is the diagonal of $A_{11}$, 
  and 
  \begin{displaymath}
    T_{\omega} :=
    \begin{bmatrix}
      I & \mathbf{0}
      \\
      -U_{22}^{-1}L_{22}^{-1}A_{21} & \mathbf{0}
    \end{bmatrix}
    \begin{bmatrix}
      I - \omega D_{11}^{-1} A_{11}
      & -\omega D_{11}^{-1}A_{12}
      \\ 
      \mathbf{0} & \mathbf{0}  
    \end{bmatrix}.
  \end{displaymath}
\end{definition}

To derive (\ref{eq:blockSmoother}), 
 we first apply the weighted Jacobi 
 to the first equation in (\ref{eq:PoissonSplitDiscretization}),
 \begin{equation}
   \label{eq:blockSmoother-1}
   \mathbf{u}^{(k+1)}_1 = 
   (I - \omega D_{11}^{-1}A_{11}) \mathbf{u}^{(k)}_1 
   + \omega D_{11}^{-1}\left(\mathbf{b}_1-A_{12}\mathbf{u}_2^{(k)}\right),
 \end{equation}
 and then exploit the LU factorization in (\ref{eq:A22_LU})
 to solve for $\mathbf{u}_2^{(k+1)}$, i.e., 
 \begin{equation}
   \label{eq:blockSmoother-2}
   L_{22}U_{22}\mathbf{u}_2^{(k+1)}
   = \mathbf{b}_2 - A_{21} \mathbf{u}^{(k+1)}_1.
 \end{equation}

After one iteration of (\ref{eq:blockSmoother}), 
 the residue vector on PLG cells,
 according to (\ref{eq:errorAndResidual}),  
 is $\mathbf{r}_2^{(k+1)} := \mathbf{b}_2
 - A_{21}\mathbf{u}_1^{(k+1)}-A_{22}\mathbf{u}_2^{(k+1)}$. 
Then (\ref{eq:blockSmoother-2}) implies 
 $\mathbf{r}_2^{(k+1)} = \mathbf{0}$. 
In other words, 
 we always have $\mathbf{r}_2^{(k+1)}=\mathbf{0}$
 for any $\mathbf{b}$;
 this is the key design of \Cref{def:blockSmoother}.

The block smoother will also be applied in \Cref{alg:VCycle,alg:FMG}
 to the residual equation $A\mathbf{e} = \mathbf{r}$.
Then the residual vector $\mathbf{r}$
 must be updated after each iteration.


In classical multigrid theory,
 the value of $\omega$ in (\ref{eq:weightedJacobi})
 is determined by 
 minimizing the supremum of the set of all damping factors 
 for high-frequency modes.
For the diagonally dominant matrix
 resulting from the second-order FD discretization of the Laplacian operator,
 it is known \cite[p. 21]{Briggs:A_Multigrid_Tutorial}
 \cite[p. 31]{Trottenberg2001}
 that the optimal value of $\omega$
 for the weighted Jacobi is $\omega = \frac{2}{3}$
 and $\omega = \frac{4}{5}$ 
 on $(0,1)$ and $(0,1)^2$, respectively. 
However,
 setting $\omega = \frac{4}{5}$ in (\ref{eq:blockSmoother})
 leads to numerical divergence in our fourth-order FV discretization, 
 even on the regular domain $(0,1)^2$.
Hence the smoothing property of the block smoother in \Cref{def:blockSmoother}
 is affected not only by the irregular domain
 but also by the fourth-order FV discretization. 
As such, it is difficult to analytically derive 
 the optimal value of $\omega$ in (\ref{eq:blockSmoother}).

In this work, 
 we set $\omega = 0.5$,
 which, according to extensive numerical experiments,
 preserves the smoothing property of the block smoother
 in (\ref{eq:blockSmoother}) 
 and minimizes the spectral radius of the two-grid correction
 operator in (\ref{eq:TG}),
 cf. \Cref{tab:spectralRadius}.

\subsection{A cut-cell V-cycle}
\label{sec:Vcycles}

Define a hierarchy of levels of cut cells
\begin{equation}
  \label{eq:grids}
  \mathfrak{C}_{\Omega}(h_f,n_l,\epsilon) := \left\{
    \mathsf{C}_{\epsilon}^{h}(\Omega) : h = h_f,\ldots,2^{n_l-1}h_f\right\},
\end{equation}
where $h_f$ is the size of the finest grid,
$n_l$ the number of levels of multigrid,
and each level $\mathsf{C}_{\epsilon}^{h}(\Omega)$ 
the output of \Cref{alg:merging}
with $(\Omega, \Omega_R, h, \epsilon)$ as the input.
By \Cref{sec:discretization},
 we have, on each level, a linear system $A^{h}\mathbf{u}^h = \mathbf{b}^h$
 in the block form of 
 (\ref{eq:PoissonSplitDiscretization}). 

\begin{algorithm}
  \caption{\textbf{V-cycle}$(A^{h}, \mathbf{u}^{h},
    \mathbf{b}^{h}, \nu_1, \nu_2)$}
  \label{alg:VCycle}

  \textbf{Input}:
  $(A^{h}, \mathbf{b}^h)$: the linear system resulting from
  discretizing (\ref{eq:ccEllipticEq}) on $\mathsf{C}_{\epsilon}^{h}$;
  \\
  $\parbox[t]{20mm}\ $
  $\mathbf{u}^{h}$: the initial guess of $(A^{h})^{-1}\mathbf{b}^h$;
  \\
  $\parbox[t]{13.2mm}\ $
  $(\nu_1, \nu_2)$: the smoothing parameters.
  \\
  \textbf{Side-effect}: $\mathbf{u}^{h}$ is updated
  as a better approximation to $(A^{h})^{-1}\mathbf{b}^h$.
  \begin{algorithmic}[1]
    \If{$h$ is the grid size of the coarsest level}
    \State $\mathbf{u}^{h}\leftarrow
    \text{\textbf{BottomSolver}}(A^{h},
    \mathbf{b}^{h})$
    \label{line:bottomSolver}
    \Else
    \For{$i = 1, \ldots ,\nu_1$}
    \State $\mathbf{u}^{h}\leftarrow
    \text{\textbf{Smooth}}(A^{h}, \mathbf{u}^{h}, \mathbf{b}^{h})$
    $\parbox[t]{11.5mm}\ $//
    \mbox{see (\ref{eq:blockSmoother})}
    \label{line:preSmooth}
    \EndFor
    \State $\mathbf{r}^{2h}\leftarrow
    \text{\textbf{Restrict}}(\mathbf{b}^{h} - A^{h} \mathbf{u}^{h})
    $\parbox[t]{15mm}\ $//
    \mbox{ see (\ref{eq:restriction})}$
    \label{line:restrict}
    \State $\mathbf{e}^{2h} \leftarrow
    \text{\textbf{V-cycle}}(A^{2h}, \mathbf{0}^{2h}, \mathbf{r}^{2h}, \nu_1, \nu_2)$
    $\parbox[t]{2.5mm}\ //\mbox{ the initial guess is a zero vector}$
    \State $\mathbf{u}^{h} \leftarrow \mathbf{u}^{h} +
    \text{\textbf{Interpolate}}(\mathbf{e}^{2h})
    $\parbox[t]{13.5mm}\ $//
    \mbox{ see~(\ref{eq:interpolation})}$
    \label{line:interpolation}
    \For{$i = 1, \ldots ,\nu_2$}
    \State $\mathbf{u}^{h}\leftarrow \text{\textbf{Smooth}}(A^{h}, \mathbf{u}^{h}, \mathbf{b}^{h})$
    $\parbox[t]{11mm}\ $//
    \mbox{see (\ref{eq:blockSmoother})}
    \label{line:postSmooth}
    \EndFor
    \EndIf
  \end{algorithmic}
\end{algorithm}

We present in \Cref{alg:VCycle}
 a cut-cell V-cycle that appears very similar
 to standard geometric multigrid V-cycles.
At line 2,
 we directly 
 solve the linear system if the current grid is the coarsest one. 
Otherwise,
 we use (\ref{eq:blockSmoother}) to block-smooth $\mathbf{u}^h$
 $\nu_1$ times at lines 4--6,
 restrict the corresponding residual to the next coarser level at line 7,
 call \Cref{alg:VCycle} recursively
 to solve the residual equation on the coarser level at line 8,
 correct the solution by the error interpolated from the coarse level
 at line 9,
 and finally block-smooth $\mathbf{u}^h$ 
 $\nu_2$ times at lines 10--12. 

For the restriction operator at line 7,
 we first observe that each irregular cell is a PLG cell
 and hence its residual becomes zero
 after one block smoothing. 
Furthermore, 
 if an irregular fine cell is covered by some coarse cell,
 then all fine cells (regular or irregular) covered by this coarse cell
 have their residuals as identically zero
 after one round of block smoothing, 
 due to the fact of the refinement ratio being 2
 and the width of SFV stencil being 5;
 see \Cref{fig:OrderOnBdry}.
Consequently, residual restriction only happens
 between \emph{regular} fine cells 
 and \emph{regular} coarse cells.
These observations obviate
 the need of volume weighting in residual restriction
 and lead to a restriction operator
 $I_{h}^{2h}: \mathbf{r}^h\rightarrow \mathbf{r}^{2h}$
 of the simple form  
\begin{equation}
  \label{eq:restriction}
  \brk{r^{2h}}_{\lfloor \frac{\bmi}{2} \rfloor} =
  2^{-\Dim} \sum\nolimits_{\bmj \in \left\{ 0, 1 \right\}^\Dim} \brk{r^{h}}_{\bmi + \bmj},
\end{equation}
where $\lfloor \mathbf{k} \rfloor$ is the greatest multi-index
less than or equal to the $\Dim$-tuple $\mathbf{k}$ of real numbers.
Thanks to the FV formulation,
 (\ref{eq:restriction}) incurs no discretization errors. 
On the other hand, the interpolation operator
$I_{2h}^{h}: \mathbf{e}^{2h}\rightarrow \mathbf{e}^{h}$
is given by 
\begin{equation}
  \label{eq:interpolation}
  \brk{e^{h}}_{\bmi} = \brk{e^{2h}}_{\lfloor \frac{\bmi}{2} \rfloor} 
\end{equation}
so that (\ref{eq:restriction}) and (\ref{eq:interpolation})
satisfy the variational property $I_{2h}^h = 2^{\Dim}(I_h^{2h})^{\top}$.

Residual restriction 
 to a coarse regular cell 
 might involve both PLG fine cells and SFV fine cells;
 for example, 
 the two PLG fine cells numbered \#14 and \#16 
 in \Cref{fig:OrderOnBdry}
 and the two SFV cells to the right of them
 are covered by a PLG coarse cell, 
 whose residual vanishes after a single pre-smoothing.
Similarly,
 after errors on coarse PLG cells
 are interpolated to fine cells, 
 those of fine PLG cells are immediately annihilated
 by one round of post-smoothing.
In addition,
 each fine PLG cell is covered by some coarse PLG cell. 
These observations, together with 
 the classical theory of geometric multigrid, 
 furnish strong heuristics in supporting the convergence 
 of the cut-cell V-cycle. 
They also suggest that
 both $\nu_1$ and $\nu_2$ be at least 1.

For the particular case of $n_l=2$ in (\ref{eq:grids}),  
 the cut-cell V-cycle reduces to a two-grid correction operator
 \cite[p. 82]{Briggs:A_Multigrid_Tutorial} given by 
 \begin{equation}
   \label{eq:TG}
   TG := T_{\omega}^{\nu_2}\left[ I -
     I_{2h}^{h}(A^{2h})^{-1}I_h^{2h}A^{h}
   \right]T_{\omega}^{\nu_1}. 
 \end{equation}

\begin{table}
  \centering
  \caption{Values of $\rho(TG)$, the spectral radius of $TG$ in (\ref{eq:TG})
    with $\omega=\frac{1}{2}$, 
    for elliptic problems on various domains
    as specified in \Cref{sec:Tests}. 
    In particular,
    the elliptic equation solved on the rotated square
    in \Cref{fig:rotatedSquare-solution}
    has a cross-derivative term 
    and the irregular boundary in \Cref{fig:squareMinusFourDisks-solnErr-N}
    is equipped with a Neumann boundary condition.
    For each case,
    we select three successively refined grids
    so that the most significant digits of the calculated spectral radii
    are the same on the two finest grids. 
    The pairs of integers in the first row are values
    of $(\nu_1,\nu_2)$. 
  }
  \label{tab:spectralRadius}
  \renewcommand{\arraystretch}{1.2}
  \begin{tabular}{c|c|cccc}
  \hline 
  Test cases
  & $(1,0)$ & $(1,1)$ & $(2,1)$ & $(2,2)$ & $(3,3)$
  \\ \hline
  the unit square $(0,1)^2$ in \Cref{fig:unitSquare-solution} & 1.080 &  0.758 &  0.603
                                                 &  0.513 & 0.378
  \\ \hline
  the rotated square $\Omega_r$ in \Cref{fig:rotatedSquare-solution} & 1.110 &  0.745 &  0.523
                                                 &  0.421 & 0.275
  \\ \hline
  $(0,1)^2$ minus a flower in \Cref{fig:squareMinusFlower-solution} & 1.069 & 0.698 &  0.483 &  0.414
                                                 & 0.283 
  \\ \hline
  $(0,1)^2$ minus four disks in
  \Cref{fig:squareMinusFourDisks-solnErr-N} & 1.272 & 0.878  &
                                                 0.641 & 0.488 & 0.308
  \\ \hline 
\end{tabular}


\end{table}%

We numerically calculate 
 the spectral radii $\rho(TG)$ of $TG$ in (\ref{eq:TG}), 
 also known as the \emph{convergence factor} of $TG$, 
 for the test problems in \Cref{sec:Tests},
 verify the independence of $\rho(TG)$ on $h$ for each test case, 
 and collect their values in Table~\ref{tab:spectralRadius}.
Before these results are discussed, 
 we mention the result in \cite[Section 4.6.1]{Trottenberg2001}
 that 0.084 is the value of 
 the convergence factor of the classical two-grid operator 
 with $(\nu_1, \nu_2)=(2,2)$,
 Gauss-Seidel smoothing, full weighting restriction, 
 and bilinear interpolation for second-order FD
 discretization of Poisson's equation
 in the unit square. 

Table~\ref{tab:spectralRadius}
 leads to observations as follows. 
First,
 $\rho(TG)$ are close to 1 for $\nu_2=0$, 
 confirming the above discussion that
 neither $\nu_1$ nor $\nu_2$ should be zero. 
Second,
 for each test case,
 $\rho(TG)$ decreases monotonically
 as $\nu_1+\nu_2$ increases,
 verifying the effectiveness of the block smoother.
Third, 
 by results of the first two test cases, 
 values of $\rho(TG)$ on the \emph{regular} domain in \Cref{fig:squares}
 are greater than those on the corresponding \emph{irregular} domain, 
 implying that 
 it is not the treatment of irregular domains
 but the fourth-order discretization
 of the elliptic operator ${\cal L}$
 and the intergrid transfer operators
 that cause $TG$ in (\ref{eq:TG})
 to be less effective than that
 of the aforementioned classical V-cycle of second-order FD discretization.
 
For $\nu_1=\nu_2=2$, 
 all values of $\rho(TG)$ are less than 0.52. 
Then it follows from $0.52^{3.79}\approx 0.084$ that, 
 to obtain the same ratio of residual reduction, 
 the number of cut-cell V-cycles
 needs to be 3.79 times as many
 as that of classical multigrid \mbox{V-cycles}.
Fortunately,
 this gap can be very much reduced
 by bringing a cut-cell FMG cycle into the big picture.

\subsection{A cut-cell FMG cycle}
\label{sec:complexity}

The convergence factor $\rho$ of a V-cycle is 
 usually independent of the grid size $h$
 and is less than 1, 
 and thus it takes $\bigO(\log(h^{-1}))$ V-cycles
 to solve the linear system $A^h\mathbf{u}^h=\mathbf{b}^h$.
It is well known from the multigrid literature
 \mbox{\cite[p. 77--78]{Briggs:A_Multigrid_Tutorial}} 
 that this suboptimal complexity of $\bigO(\log(h^{-1}))$ V-cycles
 can be improved to the optimal complexity of $\bigO(1)$ FMG cycles. 

\begin{algorithm}
  \caption{\textbf{FMG}$(A^{h}, \mathbf{r}^{h}, \nu_1, \nu_2)$}
  \label{alg:FMG}
  
  \textbf{Input}:
  $(A^{h}, \mathbf{r}^h)$: a residual equation corresponding to
  the linear system (\ref{eq:PoissonLinearDiscretization}); 
  \\
  $\parbox[t]{12.3mm}\ $
  $(\nu_1, \nu_2)$: the smoothing parameters.
  \\
  \textbf{Output}: An approximation to $(A^{h})^{-1}\mathbf{r}^{h}$.
  \begin{algorithmic}[1]
    \If{$h$ is the grid size of the coarsest level}
    \State \Return
    $\text{\textbf{BottomSolver}}(A^{h}, \mathbf{r}^{h})$
    \EndIf
    \State $\mathbf{r}^{2h}\leftarrow
    \text{\textbf{Restrict}}(\mathbf{r}^{h})$
    $\parbox[t]{20mm}\ $//
    \mbox{ see (\ref{eq:restriction})}
    \State $\mathbf{e}^{2h} \leftarrow \text{\textbf{FMG}}(A^{2h}, \mathbf{r}^{2h}, \nu_1, \nu_2)$
    $\parbox[t]{6mm}\ $//
    \mbox{ recursive call to FMG}
    \State $\mathbf{e}^{h} \leftarrow 
    \text{\textbf{Interpolate}}(\mathbf{e}^{2h})$
    $\parbox[t]{14.5mm}\ $//
    \mbox{ see~(\ref{eq:interpolation})}
    \State $\text{\textbf{VCycle}}(A^{h}, \mathbf{e}^{h},
    \mathbf{r}^{h}, \nu_1, \nu_2)$
    $\parbox[t]{11mm}\ $//
    \mbox{ see \Cref{alg:VCycle}}
    \State \Return $\mathbf{e}^{h}$ 
  \end{algorithmic}
\end{algorithm}

Our cut-cell FMG cycle is formalized in \Cref{alg:FMG}
 and illustrated in \Cref{fig:FMG-cycle}.
To solve the linear system
 $A^{h_f} \mathbf{u}^{h_f} = \mathbf{b}^{h_f}$
 on a hierarchy in (\ref{eq:grids}),
 we first convert it to a residul equation
 $A^{h_f} \mathbf{e}^{(0)} = \mathbf{r}^{(0)}$
 with an initial guess $\mathbf{u}^{(0)}$ 
 and then invoke
 FMG($A^{h_f}, \mathbf{r}^{(i)}, \nu_1, \nu_2$) iteratively.
During this iteration, 
 $\mathbf{r}^{(i)}$ is the only input parameter that changes:   
 the $i$th error $\mathbf{e}^{(i)}$ returned by FMG 
 leads to the $(i+1)$th solution
 $\mathbf{u}^{(i+1)} = \mathbf{u}^{(i)}+\mathbf{e}^{(i)}$,
 which, by (\ref{eq:errorAndResidual}),
 further yields the new residual
 $\mathbf{r}^{(i+1)} = \mathbf{b}^{h_f} - A^{h_f} \mathbf{u}^{(i+1)}$.
The iteration stops when  
 $\|\mathbf{r}^{(i)}\|$
 drops below a prescribed tolerance.
By \Cref{tab:residualReduction}, 
 one iteration of this cut-cell FMG cycle with $(\nu_1,\nu_2)=(3,3)$
 reduces the residual by a factor between 7.5 and 10 
 for numerical tests in \Cref{sec:Tests}.
 
\begin{figure}
  \center
    \begin{tikzpicture}
  \draw[dashed] (0, 3)--(1.5, 0);
  \draw[dashed] (1.5, 0)--(2, 1);
  \draw (2, 1)--(2.5, 0);
  \draw (2.5, 0)--(3, 1);
  \draw[dashed] (3, 1)--(3.5, 2);
  \draw (3.5, 2)--(4.5, 0);
  \draw (4.5, 0)--(5.5, 2);
  \draw[dashed] (5.5, 2)--(6, 3);
  \draw (6, 3)--(7.5, 0);
  \draw (7.5, 0)--(9, 3);
  \fill (0,3) circle (2pt);
  \fill (0.5,2) circle (2pt);
  \fill (1,1) circle (2pt);
  \fill (1.5,0) circle (2pt);
  \fill (2,1) circle (2pt);
  \fill (2.5,0) circle (2pt);
  \fill (3,1) circle (2pt);
  \fill (3.5,2) circle (2pt);
  \fill (4,1) circle (2pt);
  \fill (4.5,0) circle (2pt);
  \fill (5, 1) circle (2pt);
  \fill (5.5,2) circle (2pt);
  \fill (6,3) circle (2pt);
  \fill (6.5,2) circle (2pt);
  \fill (7,1) circle (2pt);
  \fill (7.5,0) circle (2pt);
  \fill (8,1) circle (2pt);
  \fill (8.5,2) circle (2pt);
  \fill (9,3) circle (2pt);
    \node[right] at (9.2, 0.1) {$\mathsf{C}_{\epsilon}^{8h}$};
    \node[right] at (9.2, 1.1) {$\mathsf{C}_{\epsilon}^{4h}$};
    \node[right] at (9.2, 2.1) {$\mathsf{C}_{\epsilon}^{2h}$};
    \node[right] at (9.2, 3.1) {$\mathsf{C}_{\epsilon}^{h}$};
  \end{tikzpicture}
  \caption{Illustrating the FMG cycle in \Cref{alg:FMG}
    on a hierarchy of four levels.
    FMG begins with a descent at line 4
    to the coarsest level $\mathsf{C}_{\epsilon}^{8h}$; 
    this is represented by the first three downward dashed line segments.
    Then the solution is
    interpolated to $\mathsf{C}_{\epsilon}^{4h}$ at line 6
    and used as the initial guess to the V-cycle at line 7
    on $\mathsf{C}_{\epsilon}^{4h}$. 
    This ``interpolation $+$ V-cycle" process is repeated recursively:
    the interpolation is represented by an upward dashed line
    and the V-cycles are represented by the solid lines.
    This FMG cycle is more effective
    than the V-cycle 
    because it comes up with a much better initial guess, 
    cf. line 8 at \Cref{alg:VCycle}.
  }
  \label{fig:FMG-cycle}
\end{figure}
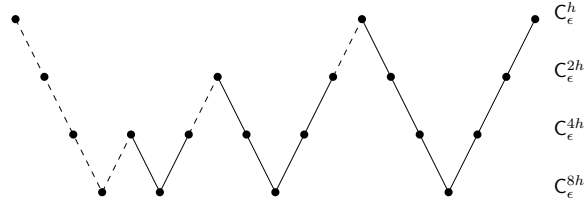

Finally, we claim that 
 the cut-cell FMG cycle in \Cref{alg:FMG} 
 is of the optimal complexity $\bigO(h^{-2})$.
In setting up the block smoother, 
 $A_{11}$ is initialized in $O(h^{-2})$ time
 while all other block matrices
 are computed in $O(h^{-1})$ time,
 cf. \Cref{thm:ComplexityBandedLUL22}.
In solving $A^{h_f} \mathbf{u}^{h_f} = \mathbf{b}^{h_f}$ on
 $\mathfrak{C}_{\Omega}(h_f)$, 
 the entire computation cost of an FMG cycle 
 is $O(h^{-2})$, the same as that of
 the deepest V-cycle,
 because an FMG cycle in 2D is at most $\frac{4}{3}$ times
 more expensive than the deepest V-cycle 
 \cite[p. 47--48]{Briggs:A_Multigrid_Tutorial}.

\section{Numerical tests}
\label{sec:Tests}

In this section
we demonstrate the fourth-order accuracy
and the optimal efficiency of our cut-cell geometric multigrid method
by results of various test problems.
To facilitate accuracy comparisons of our method
 to the second- and fourth-order EB methods
 in \cite{Johansen1998,Devendran2017}, 
 we follow \cite{Johansen1998,Devendran2017}
 to measure computational errors by the $L^p$ norms, 
\begin{equation}
  \label{eq:LpNorms}
  \Vert u \Vert_p = 
    \begin{cases}
      \left( \frac {1} {\Vert \Omega \Vert}
        \sum_{{\cal C}_{\ibold}\in \mathsf{C}_{\epsilon}^h(\Omega)}
        \left\Vert \calC_\bmi \right\Vert \cdot
        \left\vert \brk{u}_{\bmi} \right\vert^p \right)^{\frac{1}{p}}
      & \text{if } p = 1, 2; 
      \\
      \max_{{\cal C}_{\ibold}\in \mathsf{C}_{\epsilon}^h(\Omega)}
       \left\vert \brk{u}_{\bmi} \right\vert
      & \text{if } p = \infty,
    \end{cases}
\end{equation}
where $\mathsf{C}_{\epsilon}^h(\Omega)$
is the set of nonempty cut cells
in (\ref{eq:mergedCutCellSet}).
 

\subsection{A rotated square}
\label{sec:squares}

\begin{figure}
  \centering
  \subfigure[Results on the unit square] {
    \includegraphics[width=0.3\linewidth, trim=0.6in 0.29in 0.6in 0.29in, clip]{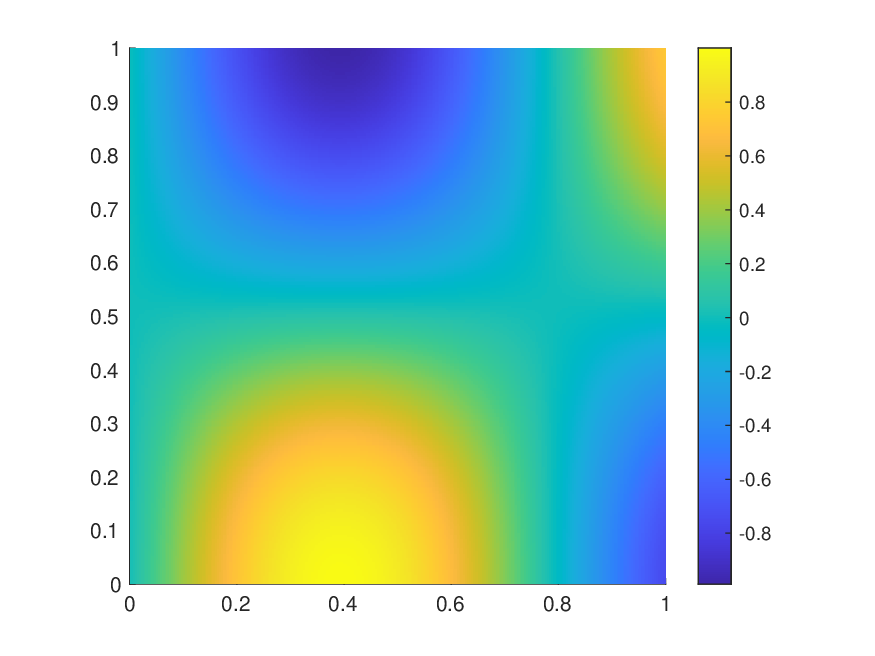}
    \label{fig:unitSquare-solution}
  }
  \hfill
  \subfigure[Results on the rotated unit square $\Omega_r$] {
    \includegraphics[width=0.41\linewidth, trim=0.6in 0.29in 0.6in 0.29in, clip]{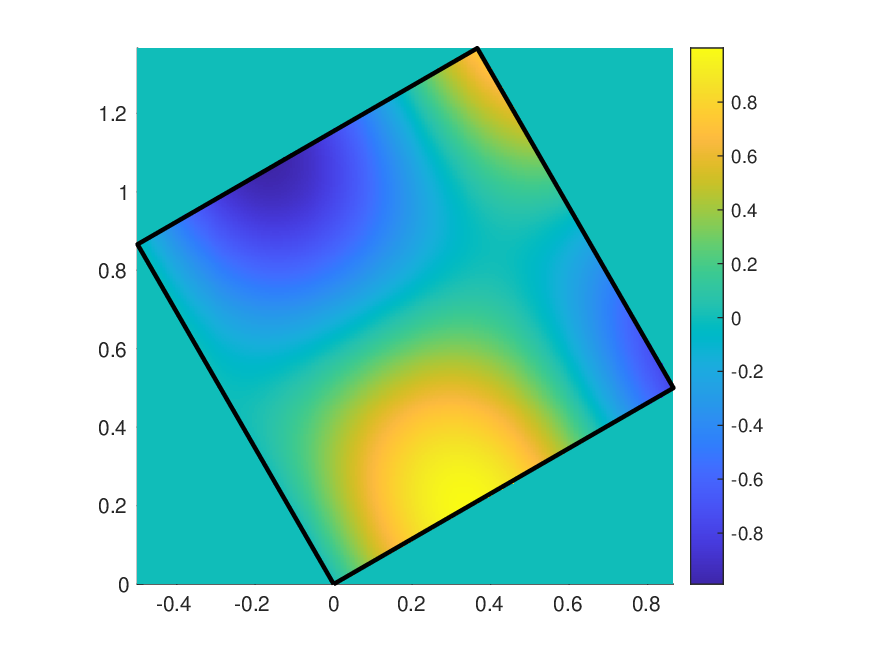}
    \label{fig:rotatedSquare-solution}
  }
  \label{fig:squares}
  \caption{
    Results of the proposed cut-cell multigrid method
    in solving the two equivalent tests in \Cref{sec:squares}
    with $h=\frac{1}{256}$.
    The two exact solutions 
    are related by a rotation of $\frac{\pi}{6}$ around $(0,0)$. 
  }
\end{figure}

\begin{table}
  \centering
  \caption{Error norms and convergence rates
    of the proposed method with $\epsilon = 0.1$
    for solving the tests in \Cref{sec:squares}.
    $\Omega_r$ is obtained by rotating $(0,1)^2$
    around the origin by $\frac{\pi}{6}$;
    see \Cref{fig:rotatedSquare-solution}. 
  }
  \begin{tabular}{c|c|ccccccc}
  \hline
  $\Omega$                    &              & $h=\frac{1}{64}$ & rate & $h=\frac{1}{128}$ & rate & $h=\frac{1}{256}$ & rate & $h=\frac{1}{512}$ \\ \hline
  \multirow{3}{*}{$(0,1)^2$}  & $L^{\infty}$ &     3.68e-08 &       4.00 &   2.30e-09 &       4.00 &   1.44e-10 &       3.99 &   9.02e-12       \\
                              & $L^{1}$      &   1.13e-08 &       4.01 &   7.00e-10 &       4.01 &   4.35e-11 &       3.91 &   2.89e-12 \\
                              & $L^2$   &     1.50e-08 &       4.01 &   9.32e-10 &       4.00 &   5.81e-11 &       3.97 &   3.71e-12 \\ \hline
  \multirow{3}{*}{$\Omega_r$} & $L^{\infty}$ &    1.35e-07 &       3.93 &   8.85e-09 &       3.93 &   5.79e-10 &       3.95 &   3.75e-11 \\
                              & $L^{1}$      &    4.83e-08 &       4.03 &   2.95e-09 &       3.91 &   1.96e-10 &       3.91 &   1.31e-11 \\
                              & $L^2$        &     5.92e-08 &       3.99 &   3.72e-09 &       3.89 &   2.51e-10 &       3.91 &   1.67e-11 \\ \hline
\end{tabular}


  \label{tab:squares}
\end{table}%

This test consists of two cases. 
First, 
we set $\Omega=(0,1)^2$ and $(a,b,c) = (1, 0, 2)$
 in (\ref{eq:ccEllipticEq}), 
 for which the exact solution is 
\begin{equation}
  \label{eq:squares}
  \forall (x_1, x_2)\in \overline{\Omega},\quad 
  u(x_1, x_2) = \sin(4 x_1) \cos(3 x_2),
\end{equation}
and the boundary condition is the Dirichlet condition from
(\ref{eq:squares}).
Due to the regularity of $\Omega$, 
 all cut cells in $\mathsf{C}_{\epsilon}^h(\Omega)$ are SFV cells, 
 the blocks $A_{21}$, $A_{12}$, and $A_{22}$ vanish, 
 and the linear system (\ref{eq:PoissonSplitDiscretization})
 reduces to that of the standard fourth-order FV discretization of
 (\ref{eq:ccEllipticEq}). 
Also, the block smoother in \Cref{def:blockSmoother}
 reduces to the weighted Jacobi. 
 
In the second case, 
 the domain $\Omega_r$ is obtained by rotating the unit square
 around the origin by $\frac{\pi}{6}$; 
 see \Cref{fig:rotatedSquare-solution}.
For $\left(a, b, c\right) = \frac{1}{4}\left(5, -2\sqrt{3}, 7\right)$, 
 the exact solution $u: \overline{\Omega} \to \mathbb{R}$
 of (\ref{eq:ccEllipticEq})
 is obtained by rotating that in (\ref{eq:squares})
 by $\frac{\pi}{6}$. 
A Dirichlet condition is imposed to ensure
 that the only difference of these two cases
 is the regularity of the boundary.
The main goal of this setup is to examine
 how the cross-derivative term
 and the PLG discretizations affect the solution errors. 
 
For the two equivalent systems, 
 numerical solutions with $h=\frac{1}{256}$ 
 are shown in Figure~\ref{fig:squares}, 
 with error norms and convergence rates 
 listed in Table~\ref{tab:squares}, 
 where the fourth-order accuracy 
 are clearly demonstrated. 
Each error norm on the irregular domain 
 is greater than its counterpart on the regular domain, 
 due to the PLG discretization
 and the larger SFV stencil for the additional cross-derivative term.
However, 
 the ratio of the two error norms is bounded by 4.6
 and we consider the slightly lower accuracy as a reasonable cost 
 for PLG and the cross-derivative term.

\subsection{A square minus a flower}
\label{sec:squareMinusFlower}

\begin{figure}
  \centering
  \subfigure[the numerical solution] {
    \includegraphics[width=0.305\linewidth, trim=0.6in 0.28in 0.6in 0.3in, clip]{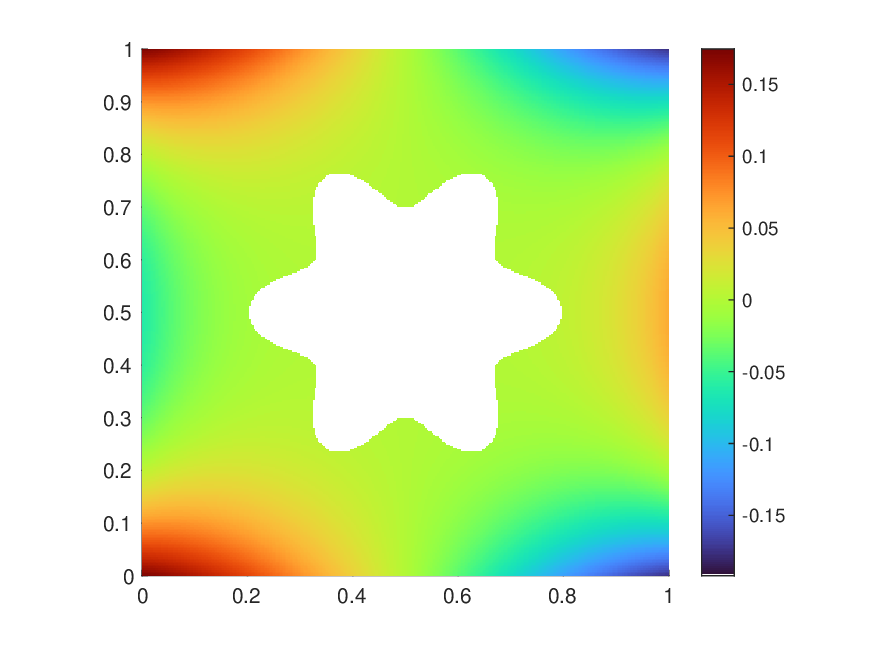}
    \label{fig:squareMinusFlower-solution}
  }
  \hfill
  \subfigure[the truncation error] {
    \includegraphics[width=0.305\linewidth, trim=0.6in 0.28in 0.6in 0.3in, clip]{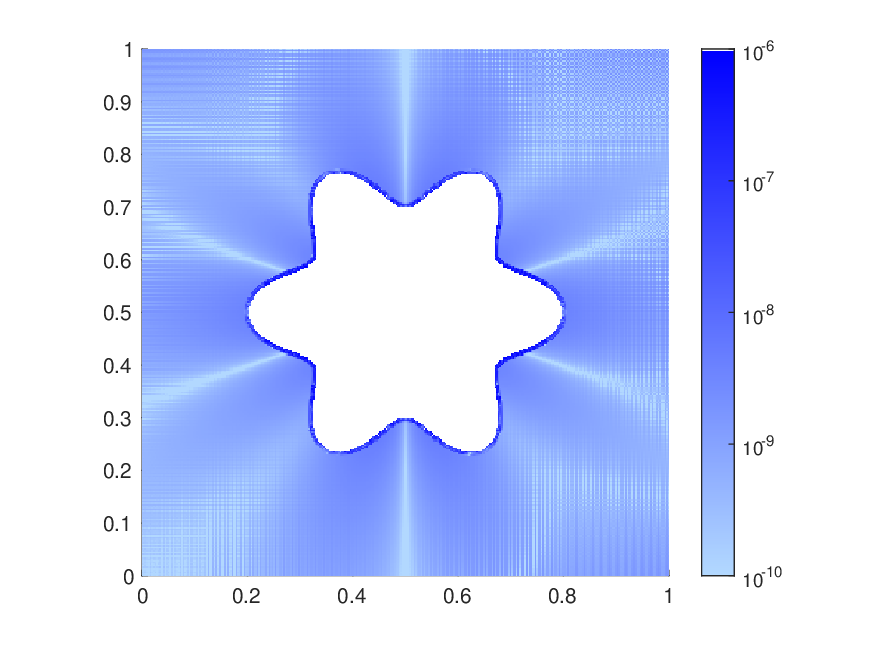}
    \label{fig:squareMinusFlower-truncErr}
  }
  \hfill
  \subfigure[the solution error] {
    \includegraphics[width=0.29\linewidth, trim=0.6in 0.28in 0.6in 0.3in, clip]{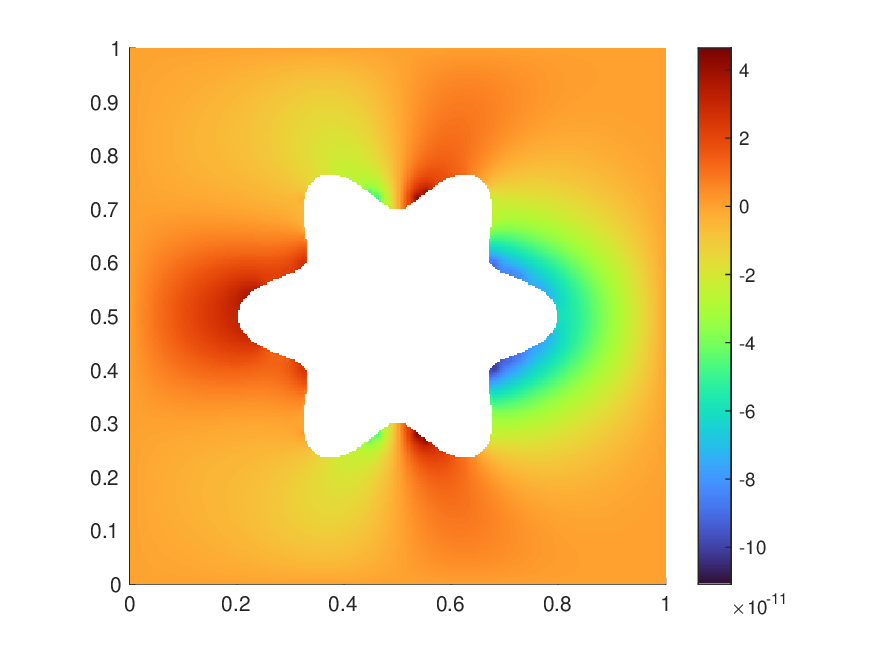}
    \label{fig:squareMinusFlower-solnErr}
  }
  \caption{
    Results of the proposed method
    in solving the problem in \Cref{sec:squareMinusFlower}
    with $h=\frac{1}{80}$.
    Different from subplots (a,c),
    subplot (b) has a logarithmic scale for representing truncation errors.
  }
  \label{fig:squareMinusFlower}
\end{figure}

\begin{table}
  \centering
  \caption{Truncation and solution errors
    of the proposed cut-cell method with $\epsilon = 0.02$
    and a second-order EB method~\cite{Johansen1998}
    for solving the test problem in \Cref{sec:squareMinusFlower}.
  }            
  \begin{tabular}{cccccccc}
  \hline
  \multicolumn{8}{c}{Truncation errors of the EB method by Johansen and Colella~\cite{Johansen1998}} \\
  \hline
  & $h=\frac{1}{40}$  & rate
  & $h=\frac{1}{80}$  & rate
  & $h=\frac{1}{160}$ & rate
  & $h=\frac{1}{320}$                                                                     \\
  \hline
  $L^\infty$ & 1.66e-03        & 2.0  & 4.15e-04 & 2.0  & 1.04e-04 & 2.0  & 2.59e-05         \\
  \hline
  \multicolumn{8}{c}{Truncation errors of the proposed fourth-order cut-cell method}
  \\
  \hline
  & $h=\frac{1}{40}$  & rate
  & $h=\frac{1}{80}$  & rate
  & $h=\frac{1}{160}$ & rate
  & $h=\frac{1}{320}$                                                                     \\
  \hline
  $L^\infty$ &   6.53e-04 &       3.02 &   8.03e-05 &       2.49 &   1.42e-05 &       3.28 &   1.47e-06 \\
  $L^1$ &   1.48e-05 &       4.01 &   9.23e-07 &       4.02 &   5.68e-08 &       4.11 &   3.29e-09 \\
  $L^2$ &   5.27e-05 &       3.60 &   4.35e-06 &       3.52 &   3.79e-07 &       3.64 &   3.05e-08 \\
  \hline
  \hline
  \multicolumn{8}{c}{Solution errors of the EB method by Johansen and Colella~\cite{Johansen1998}} \\
  \hline
  & $h=\frac{1}{40}$  & rate
  & $h=\frac{1}{80}$  & rate
  & $h=\frac{1}{160}$ & rate
  & $h=\frac{1}{320}$                                                                   \\
  \hline
  $L^\infty$ & 4.78e-05        & 1.85 & 1.33e-05 & 1.98 & 3.37e-06 & 1.95 & 8.72e-07       \\
  \hline
  \multicolumn{8}{c}{Solution errors of the proposed fourth-order cut-cell method}                                        \\
  \hline
  & $h=\frac{1}{40}$  & rate
  & $h=\frac{1}{80}$  & rate
  & $h=\frac{1}{160}$ & rate
  & $h=\frac{1}{320}$                                                                   \\
  \hline
  $L^\infty$ &   5.42e-07 &       4.98 &   1.72e-08 &       3.73 &   1.29e-09 &       3.68 &   1.01e-10 \\
  $L^1$ &   7.72e-08 &       5.23 &   2.06e-09 &       3.86 &   1.42e-10 &       3.84 &   9.90e-12 \\
  $L^2$ &   1.31e-07 &       5.20 &   3.54e-09 &       3.89 &   2.39e-10 &       3.84 &   1.66e-11 \\
  \hline
\end{tabular}


  \label{tab:squareMinusFlower}
\end{table}%

In this test, we follow \cite[Problem 3]{Johansen1998}
to solve Poisson's equation
on an irregular domain $\Omega = R \cap \Omega_1$,
where $R = (-0.5,0.5)^2$,
$\Omega_1 = \left\{ \left( r, \theta \right)
  : r > 0.25 + 0.05 \cos 6 \theta \right\}$, 
and $\left( r, \theta \right)$ are the polar coordinates satisfying
$(x_1, x_2) = (r \cos \theta, r \sin \theta)$.
As shown in \Cref{fig:squareMinusFlower-solution},
 we set the exact solution as
 \begin{equation}
   \label{eq:squareMinusFlower}
  \forall (x_1,x_2)  \in\overline{\Omega},\quad
  u(x_1,x_2) = u(r, \theta) = r^4 \cos 3 \theta 
\end{equation}
 and impose Dirichlet and Neumann conditions
 on $\partial R$ and $\partial \Omega_1$, respectively. 

Due to the symmetric FV formulas in \Cref{sec:symmetricFD-cells}, 
 the truncation error $\tau_{\ibold}$
 for an SFV cell ${\cal C}_{\ibold}$
 is $O(h^4)$. 
For a PLG cell ${\cal C}_{\ibold}$, however, 
 it follows from (\ref{eq:goalOfApproximation})
 and the opening paragraph of \Cref{sec:discreteElliptic}
 that the truncation error for the $\ibold$th cut cell
 is given by
 $\tau_{\ibold}:={\beta}_{\min}^{\top} \overline{\mathbf{u}} - \brk{\calL u}_\bmi
 = O(h^3)$. 
This is confirmed both in \Cref{fig:squareMinusFlower-truncErr}
 and \Cref{tab:squareMinusFlower},
 where the convergence rates of truncation errors
 are asymptotically close to 3, 3.5, and 4 
 in the $L^\infty$, $L^2$, and $L^1$ norms, respectively.
In \Cref{fig:squareMinusFlower-solnErr}, 
 the non-uniformness of truncation errors
 causes solution errors to be oscillatory; however,
 the magnitude of solution errors
 is very small ($\sim 10^{-10}$)
 even for the large grid size $h=\frac{1}{80}$. 
More importantly,
 the large truncation errors near the boundary
 do not affect the fourth-order accuracy
 of solution errors; 
 this is well known for FD/FV methods
 and is confirmed in \Cref{tab:squareMinusFlower}.
 

Truncation errors and solution errors of 
 the classical second-order EB method
 by Johansen and Colella~\cite{Johansen1998}
 are also listed in \Cref{tab:squareMinusFlower}. 
Clearly,
 our method is much more accurate: 
 the $L^\infty$ solution error of our method
 on the coarest grid of $h=\frac{1}{40}$ 
 is smaller than that of the second-order EB method
 on the finest grid of $h=\frac{1}{320}$.


\subsection{A square minus four disks}
\label{sec:squareMinusFourDisks}

Consider a problem in \cite[\S 5.2]{Devendran2017}
of solving Poisson's equation on the domain
$\Omega = R \setminus \Omega_d$, 
where $R=(0, 1)^2$
and $\Omega_d$ is the closure of the union of four disks,
whose centers and radii $(c_1,c_2; r)$ are
$(0.5,0.5;0.2)$, $(0.5,0.735;0.1)$, $(0.2965,0.3825;0.1)$,
and $(0.7035,0.3825,0.1)$.
At each of the six kinks on $\Omega_d$, 
 a level-set function that implicitly represents
 $\partial\Omega_d$ would be $C^1$ discontinuous.
Following \cite[\S 5]{Devendran2017}, 
 we set the exact solution as 
\begin{equation}
  \label{eq:squareMinusFourDisks}
  \forall (x_1, x_2)\in \overline{\Omega},\quad
  u = \sin(\pi x_1)\sin(\pi x_2)
\end{equation}
and impose on $\partial R$
a Dirichlet condition from (\ref{eq:squareMinusFourDisks}).

In \Cref{tab:squareMinusFourDisks},
 error norms and convergence rates of our cut-cell method
 and the fourth-order EB method in \cite{Devendran2017}
 are presented for solving this test
 with Dirichlet and Neumann conditions on $\partial\Omega_d$.
The fourth-order EB method performs poorly:
 its converge rates barely reach 2 and 1
 for Dirichlet and Neumann conditions, respectively; 
 as shown in \cite[Fig. 8]{Devendran2017}, 
 its largest solution errors concentrate around the six kinks. 
This is not surprising because, as discussed in \Cref{sec:intro}, 
 the error of the normal vector near a kink is $O(1)$
 and thus the integral of fluxes over faces of an irregular cut cell
 is calculated with an error of $O(h)$.
At the presence of kinks,
 this accuracy deterioration is unavoidable
 if the discretization of (\ref{eq:ccEllipticEq})
 is coupled with an implicit representation of the domain boundary.
Although the fourth-order accuracy can be recovered
 by smoothing the geometric description,
 this mollification process requires substantial extra care
 and its effectiveness depends largely on 
 mollification formulas and the nature of the governing equation
 \cite[\S 6]{Devendran2017}.
 
\begin{table}
  \centering
  \caption{Solution errors and convergence rates
    of the proposed cut-cell method with $\epsilon = 0.08$
    and a fourth-order EB method~\cite{Devendran2017}
    in solving the test problem in \Cref{sec:squareMinusFourDisks}. 
  }
  \setlength{\tabcolsep}{9pt}
\begin{tabular}{cccccccc}
  \hline
  & $h=\frac{1}{64}$ & rate & $h=\frac{1}{128}$ & rate & $h=\frac{1}{256}$ & rate & $h=\frac{1}{512}$    \\
  \hline
  \multicolumn{8}{c}{4th-order EB method
  without mollifying kinks; a Dirichlet condition on $\partial \Omega_d$}                                            \\
  \hline
  $L^{\infty}$ & 2.80e-03 & $-3.05$    & 2.32e-02 & 3.10     & 2.71e-03 & 2.06     & 6.50e-04 \\
  $L^{1}$ & 3.23e-05 & 2.82    & 4.56e-06 & 4.49     & 2.01e-07 & 1.76        & 5.94e-08 \\
  $L^{2}$ & 1.09e-04 & 0.56        & 7.40e-05 & 3.90        & 4.97e-06 & 1.13        & 2.26e-06            \\
  \hline
  \multicolumn{8}{c}{4th-order EB method with kink mollification; a Dirichlet condition on $\partial \Omega_d$} \\
  \hline
  $L^{\infty}$ & 2.64e-08 & 4.06    & 1.58e-09 & 4.03     & 9.65e-11 & 3.85     & 6.67e-12 \\
  $L^{1}$ & 1.08e-08 & 4.08     & 6.38e-10 & 4.03     & 3.88e-11 & 3.84     & 2.69e-12 \\
  $L^{2}$ & 1.34e-08 & 4.11     & 7.76e-10 & 4.05     & 4.68e-11 & 3.86     & 3.23e-12 \\
  \hline
  \multicolumn{8}{c}{the proposed cut-cell method; a Dirichlet condition on $\partial \Omega_d$}\\
  \hline
  $L^\infty$ &   5.44e-08 &       5.48 &   1.22e-09 &       3.88 &   8.28e-11 &       4.52 &   3.62e-12 \\
  $L^1$ &   9.50e-09 &       4.87 &   3.26e-10 &       4.41 &   1.54e-11 &       4.19 &   8.41e-13 \\
  $L^2$ &   1.17e-08 &       4.94 &   3.81e-10 &       4.41 &   1.79e-11 &       4.19 &   9.84e-13 \\
  \hline
  \hline
  & $h=\frac{1}{64}$ & rate & $h=\frac{1}{128}$ & rate & $h=\frac{1}{256}$ & rate & $h=\frac{1}{512}$    \\
  \hline
  \multicolumn{8}{c}{4th-order EB method without mollifying kinks; a Neumann condition on $\partial \Omega_d$}
  \\
  \hline
  $L^{\infty}$ & 3.10e-02 & 0.14  & 2.82e-02 & 0.52  & 1.97e-02 & 0.38  & 1.52e-02 \\
  $L^{1}$   & 1.42e-03 & $-0.58$ & 2.13e-03 & 0.62  & 1.38e-03 & 1.03   & 6.75e-04 \\
  $L^{2}$   & 2.72e-03 & $-0.24$ & 3.20e-03 & 0.67  & 2.02e-03 & 0.96   & 1.04e-03 \\
  \hline
  \multicolumn{8}{c}{4th-order EB method with kink mollification; a Neumann condition on $\partial \Omega_d$}
  \\
  \hline
  $L^{\infty}$ & 1.84e-07 & 4.02 & 1.13e-08 & 3.97 & 7.21e-10 & 3.74 & 5.39e-11 \\
  $L^{1}$ & 5.83e-08 & 3.94 & 3.78e-09 & 3.97 & 2.40e-10 & 3.86 & 1.65e-11 \\
  $L^{2}$ & 7.35e-08 & 3.96 & 4.73e-09 & 3.98 & 2.99e-10 & 3.87 & 2.04e-11 \\
  \hline
  \multicolumn{8}{c}{the proposed 4th-order cut-cell method; a Neumann condition on $\partial \Omega_d$}
  \\ \hline
  $L^\infty$ &   2.07e-07 &       4.29 &   1.06e-08 &       4.01 &   6.56e-10 &       3.79 &   4.76e-11 \\
  $L^1$ &   2.80e-08 &       5.02 &   8.65e-10 &       3.95 &   5.58e-11 &       3.79 &   4.04e-12 \\
  $L^2$ &   3.98e-08 &       4.63 &   1.61e-09 &       4.09 &   9.43e-11 &       3.75 &   7.03e-12 \\
  \hline
\end{tabular}


  \label{tab:squareMinusFourDisks}
\end{table}%

\begin{figure}
  \centering
  \subfigure[A Neumann condition on the disks $\partial \Omega_d$] {
    \includegraphics[width=0.45\linewidth, trim=0.6in 0.30in 0.59in 0.1in, clip]{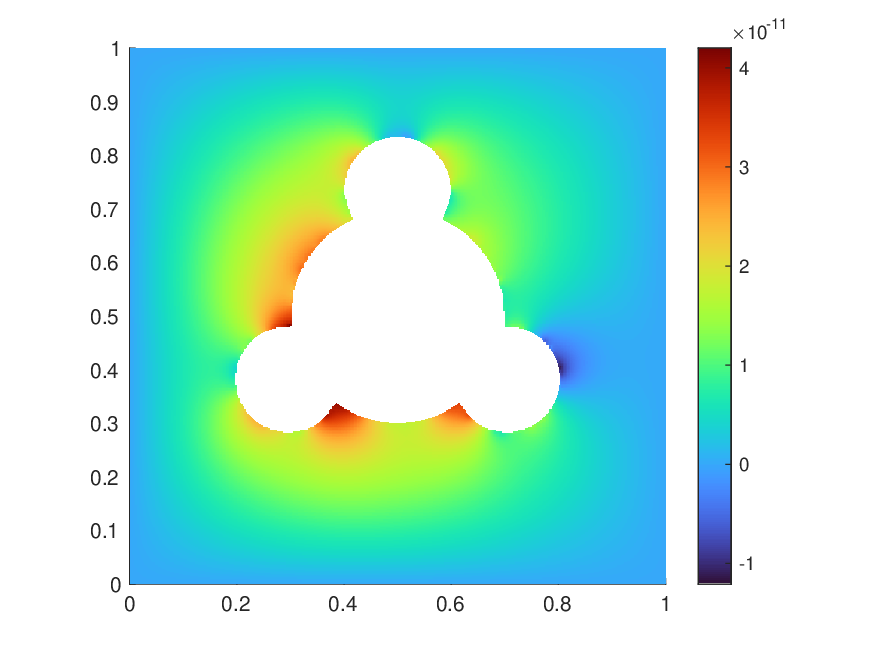}
    \label{fig:squareMinusFourDisks-solnErr-N}
  }
  \subfigure[A Dirichlet condition on the disks $\partial \Omega_d$] {
    \includegraphics[width=0.45\linewidth, trim=0.6in 0.28in 0.59in 0.1in, clip]{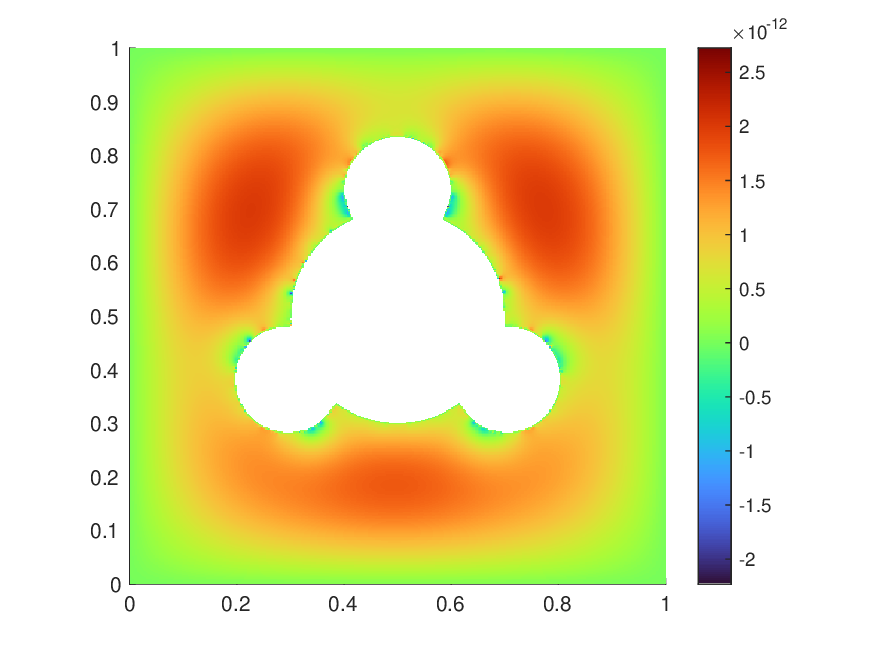}
    \label{fig:squareMinusFourDisks-solnErr-D}
  }
  \caption{Solution errors of the proposed cut-cell method
    in solving the problem in \Cref{sec:squareMinusFourDisks}
    with $h=\frac{1}{512}$.
    A Dirichlet condition is applied on the square boundary.
    All boundary conditions are derived from
    (\ref{eq:squareMinusFourDisks}). 
  }
  \label{fig:squareMinusFourDisks}
\end{figure}

In comparison,
 convergence rates of our cut-cell method
 are closed to 4 in both cases
 and its solution errors in all norms
 are smaller than those of the fourth-order EB method
 with mollifications.
As shown in \Cref{fig:squareMinusFourDisks}, 
 solution errors of our cut-cell method
 are not concentrated at the six kinks.
This is also unsurprising because 
 (i) the explicit representation of domain boundary
 by cubic splines admits a fourth- and higher-order
 approximation of the geometry of any irregular cut cell
 and (ii)
 the integrals of solutions over an irregular cut cell
 can be approximated to very high-order accuracy
 by Green's theorem and Gauss quadrature formulas.
In summary,
 the integral formulation of our cut-cell method
 (enabled by explicit representation of geometry)
 is advantageous over the differential formulation
 of previous EB methods.
 


\subsection{A panda}
\label{sec:panda}

To showcase the capability of the proposed cut-cell method
 in handling complex topology and geometry,
 we numerically solve (\ref{eq:ccEllipticEq})
 on the domain of a panda shown in Figure
 \ref{fig:panda-solution},  
 which is adapted from that in \cite[Figure 10]{Zhang2020:YinSets}
 with a sufficent number of breakpoints. 
The same spline representation of the panda boundary 
 is used for all grid sizes. 
The complex topology and geometry of the panda
 pose significant challenges to 
 a numerical solver. 

\begin{figure}
  \centering
  \subfigure[the numerical solution] {
    \includegraphics[width=0.43\linewidth, trim=0.68in 0.25in 0.7in 0.25in, clip]{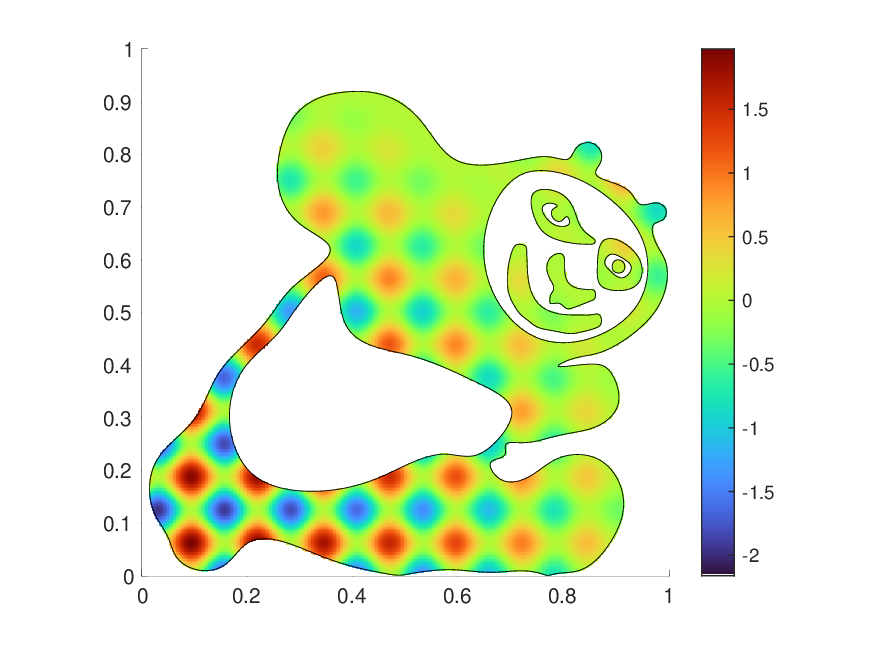}
    \label{fig:panda-solution}
  }
  \hfill
  \subfigure[the solution error] {
    \includegraphics[width=0.43\linewidth, trim=0.68in 0.25in 0.7in 0.25in, clip]{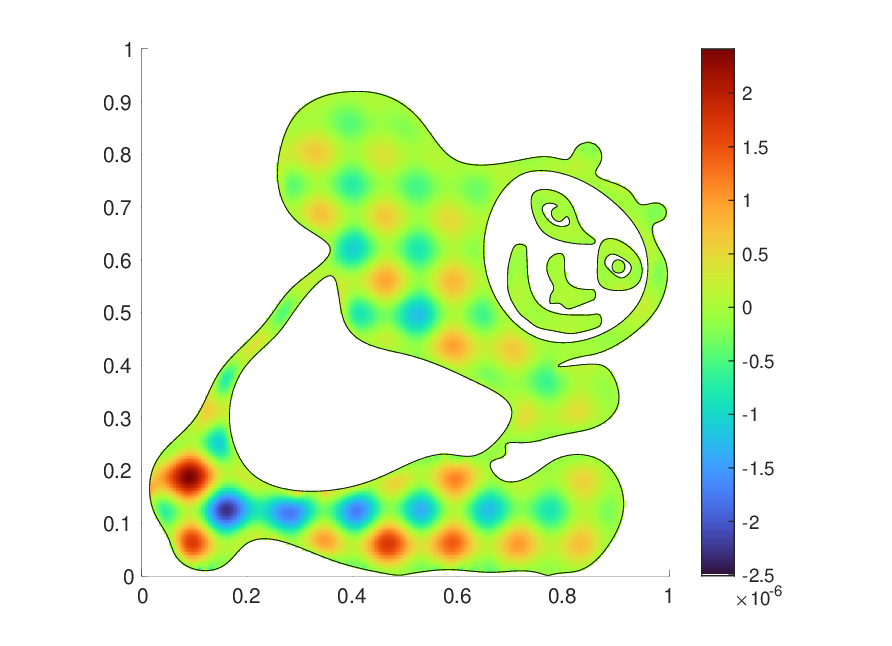}
    \label{fig:panda-solnErr}
  }
  \caption{Results of the proposed cut-cell method
    in solving 
    the panda problem in \Cref{sec:panda}
    with $(a,b,c) = (1,1,2)$ and $h=\frac{1}{512}$.}
  \label{fig:panda}
\end{figure}

\begin{table}
  \centering
  \caption{Solution errors and convergence rates
    of the proposed cut-cell method with $\epsilon = 0.01$
    in solving the panda problem in \Cref{sec:panda},
    where we impose the Dirichlet condition (\ref{eq:panda})
    on the boundary.
  }
  \begin{tabular}{c|c|ccccccc}
  \hline
  $(a,b,c)$                    &              & $h=\frac{1}{256}$  & rate
  & $h=\frac{1}{512}$  & rate
  & $h=\frac{1}{1024}$ & rate
  & $h=\frac{1}{2048}$ \\ \hline
  \multirow{3}{*}{$(1,1,2)$}  & $L^{\infty}$ & 7.97e-05           & 5.05     & 2.41e-06
                       & 4.03               & 1.48e-07 & 4.02     & 9.14e-09 \\
                               & $L^{1}$      & 5.00e-06           & 3.97     & 3.18e-07
                       & 3.96               & 2.05e-08 & 3.97     & 1.31e-09 \\
                               & $L^2$   & 7.55e-06           & 3.96     & 4.87e-07
                       & 3.96               & 3.13e-08 & 3.98     & 1.99e-09 \\\hline
  \multirow{3}{*}{$(1,0,2)$} & $L^{\infty}$ &  5.47e-05           & 4.58     & 2.29e-06
                       & 4.01               & 1.42e-07 & 4.01     & 8.82e-09 \\
                               & $L^{1}$      & 4.75e-06           & 3.94     & 3.09e-07
                       & 3.95               & 2.00e-08 & 3.97     & 1.28e-09 \\
                               & $L^2$        &7.25e-06           & 3.93     & 4.76e-07
                       & 3.95               & 3.08e-08 & 3.97     & 1.96e-09 \\ \hline
\end{tabular}


  \label{tab:panda}
\end{table}%

The exact solution of this test is
\begin{equation}
  \label{eq:panda}
  \forall (x_1, x_2)\in \overline{\Omega},\quad
  u(x_1,x_2) = \left(x_1^2+x_2^2-1\right) [\sin(50x_1)+\cos(50x_2)]. 
\end{equation}
We impose the corresponding Dirichlet condition
 on the boundary of the panda
 and numerically solve (\ref{eq:ccEllipticEq}) 
 for $(a,b,c) = (1,1,2)$ and $(1,0,2)$. 
Results of our cut-cell method for the case
 with the cross-derivative term
 on the grid of $h=\frac{1}{512}$
 are plotted in Figure \ref{fig:panda}
 and the error norms 
 are listed in Table \ref{tab:panda},
 where the convergence rates are very close to 4.0 in all norms,
 demonstrating the fourth-order accuracy
 and the capability of our method in handling complex domains.
In addition, 
 quantitative results for $(a,b,c) = (1,1,2)$
 are very close to those for $(a,b,c) = (1,0,2)$, 
 indicating that the cross-derivative term is handled satisfactorily.
 
\begin{table}
  \centering
  \caption{
    CPU time in seconds for solving the panda test in
    \Cref{sec:panda} with $(a,b,c) = (1,1,2)$
    on an AMD Threadripper PRO 3975WX at 4.0Ghz. 
    For each of the four setup steps, 
    we also report 
    its percentage of the entire cost of setup
    in a pair of parentheses. 
  }
  \begin{tabular}{|c|c|c|c|c|}
  \hline
  Stages & steps  & $h=\frac{1}{512}$ & $h=\frac{1}{1024}$
  & $h=\frac{1}{2048}$
  \\ \hline
  \multirow{5}{*}{\vspace{-22mm}Setup} & \makecell[t]{generate the set $\mathsf{C}_{\epsilon}^h(\Omega)$\\
  of cut cells by \Cref{alg:merging}}
  & \makecell[t]{0.081\\ (3.7\%)}  & \makecell[t]{0.303 \\ (6.7\%)}  & \makecell[t]{1.12 \\ (11.7\%)} \\ \cline{2-5} 
  & \makecell[t]{locate a poised lattice \\ for each PLG cell}
  & \makecell[t]{0.087\\ (3.9\%)}  & \makecell[t]{0.180\\ (4.0\%)}
  & \makecell[t]{0.407 \\ (4.2\%)}  \\ \cline{2-5} 
  & \makecell[t]{determine the linear system
  \\ in (\ref{eq:PoissonSplitDiscretization}) by steps in \Cref{sec:discretization}}
  & \makecell[t]{2.01\\ (90.8\%)} & \makecell[t]{3.99\\ (87.8\%)}
  & \makecell[t]{7.94 \\ (82.3\%)} \\ \cline{2-5} 
  & \makecell[t]{compute the block smoother \\ in
  \Cref{def:blockSmoother}}
  & \makecell[t]{0.036\\ (1.6\%)}
  & \makecell[t]{0.073 \\ (1.5\%)}  & \makecell[t]{0.176 \\ (1.8\%)}  \\
  \cline{2-5} 
  & \makecell[t]{the entire cost of setup, i.e., 
  \\ the sum of the above four steps}
  & \makecell[t]{2.21\\ (100\%)}  & \makecell[t]{4.55 \\ (100\%)}
  & \makecell[t]{9.64\\ (100\%)}
  \\\hline
  \multirow{2}{*}{Solve}      \
  & \makecell[t]{block smoothing only}  & \makecell[t]{0.521}          & \makecell[t]{2.38}           & \makecell[t]{9.84}           \\ \cline{2-5} 
  & \makecell[t]{the entire cost of FMG cycles}                 & \makecell[t]{1.08}          & \makecell[t]{3.63}           & \makecell[t]{14.3}           \\ \hline
\end{tabular}


  \label{tab:timeConsumption}
\end{table}

Computational costs of the main components
 of our cut-cell method
 are reported in \Cref{tab:timeConsumption}
 for the case of $(a,b,c) = (1,1,2)$.
The generation of cut cells by \Cref{alg:merging}
 clearly has the $O(h^{-2})$ complexity. 
In contrast, 
 all other setup steps have the optimal $O(h^{-1})$ complexity,
 confirming the analysis in \Cref{sec:discretization}
 and \Cref{sec:order_L22,sec:smoother}. 
In particular,
 the complexity of determining the linear system (\ref{eq:PoissonSplitDiscretization})
 is only $O(h^{-1})$ 
 because the block $A_{11}$ is never assembled
 but applied ``on the fly'' inside the weighted Jacobi.
As indicated by the last row of \Cref{tab:timeConsumption},
 FMG cycles have the optimal complexity of $O(h^{-2})$, 
 which confirms our analysis in \Cref{sec:complexity}. 
 
The cost of generating cut cells is very much dominated
 by that of determining the linear system
 (\ref{eq:PoissonSplitDiscretization}),
 which holds even on the finest grid.
Consequently,
 the cost of the entire initial setup displays
 a roughly linear growth as the grid size $h$ is reduced.
Being the most expensive component of V-cycles and FMG cycles, 
 block smoothing consumes more CPU time
 than the initial setup on the finest grid, 
 since the $O(h^{-2})$ growth of its cost 
 is higher than the linear growth.

\subsection{FMG efficiency}
\label{sec:efficiency}

The panda is a good representative of complex domains
 while the rotated square in \Cref{fig:rotatedSquare-solution}
 that of the other extreme of irregular but simple domains.
For the rotated square in \Cref{fig:rotatedSquare-solution}
 with $h=\frac{1}{256}, \frac{1}{512}, \frac{1}{1024}$, 
 we record computational costs (not shown) of the main components. 
The cost of generating cut cells
 grows quadratically
 with respect to the reduction of $h$
 while those of all other steps in the initial setup grows linearly;
 the $O(h^{-2})$ complexity of FMG cycles are also confirmed.
In addition, the consumed CPU time in seconds for $h=\frac{1}{1024}$
 is 1.89, 2.62, 3.57, and 5.31
 for the determination of (\ref{eq:PoissonSplitDiscretization}),
 the entire initial setup, the block smoothing,
 and all FMG cycles, respectively. 
The cost ratio of block smoothing over all FMG cycles
is $\frac{3.57}{5.31}\approx 0.67$ for the rotated square,
 which is very close to that
 ($\frac{2.38}{3.63}\approx 0.66$) for the panda, cf. \Cref{tab:timeConsumption}.
Due to the simple geometry of the rotated square, 
 the cost ratio of the entire setup over all FMG cycles, 
 $\frac{2.62}{5.31}\approx 0.49$, 
 is much smaller than that ($\frac{4.55}{3.63}\approx 1.25$)
 for the panda test.
We sum up main conclusions of the above discussions as follows.
\begin{itemize}
\item The proposed cut-cell method has the optimal
  complexity of $O(h^{-2})$.
\item On a coarse grid,
  the initial setup might be more expensive than the FMG cycles.
  However, there exists a grid size $h^*$ such that, 
  for any $h<h^*$,
  the computational cost of the initial setup is less than 
  that of the FMG cycles.
\item Block smoothing consumes $\frac{2}{3}$ of the entire CPU time
  of FMG cycles and is asymptotically the most expensive component
  of the proposed method.
\end{itemize}
 
\begin{figure}
  \centering
  \subfigure[\Cref{sec:squares} with $h = \frac{1}{512}$]{
    \includegraphics[width=0.4\linewidth, trim=0.1in 0.16in 0.05in 0.08in, clip]{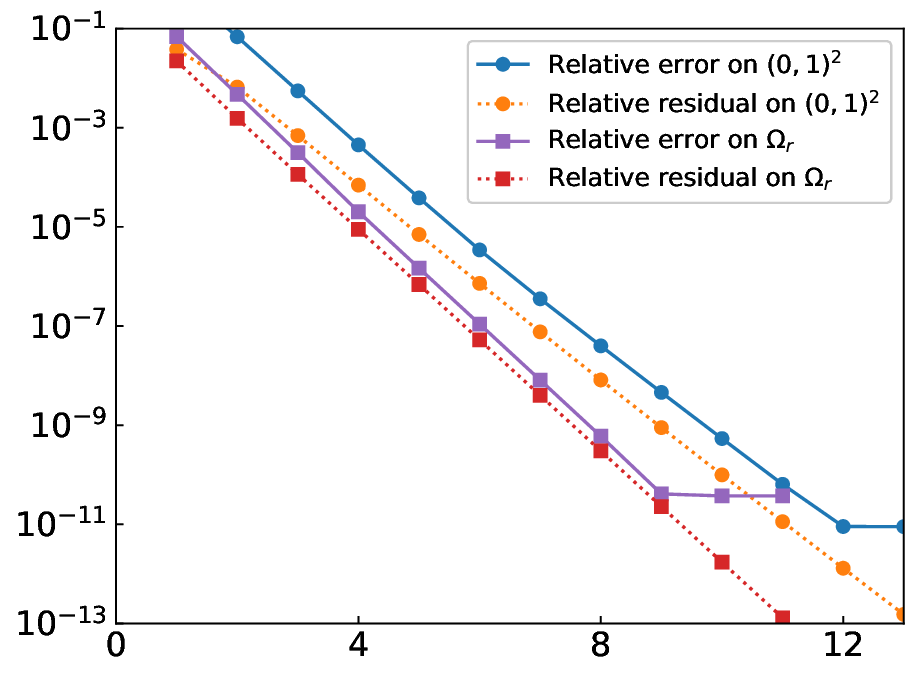}
    \label{fig:residualReduction-squares}
  }
  \hfill
  \subfigure[\Cref{sec:panda} with $h = \frac{1}{2048}$]{
    \includegraphics[width=0.4\linewidth, trim=0.1in 0.16in 0.05in 0.08in, clip]{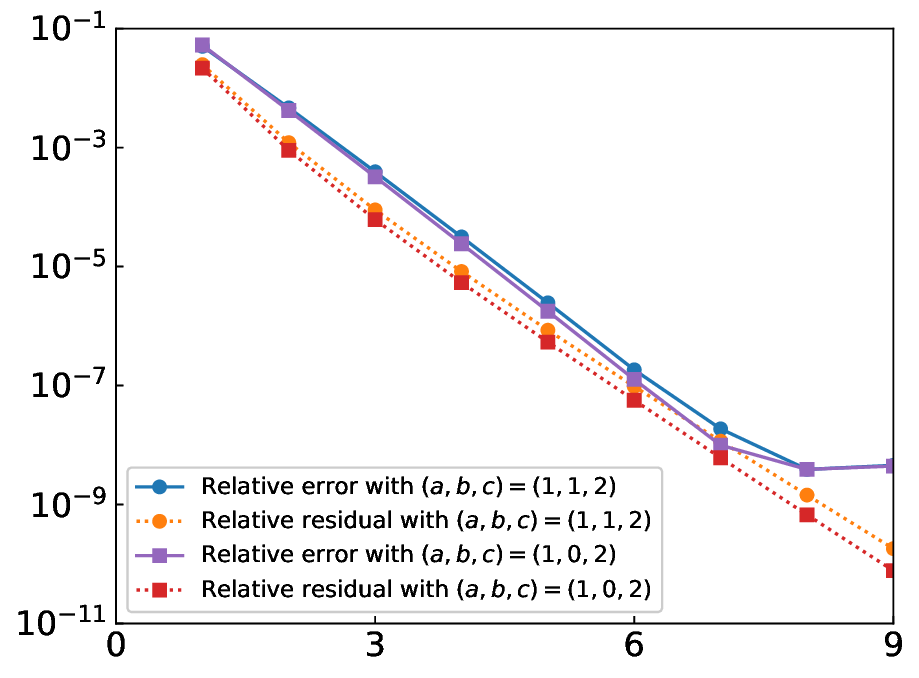}
    \label{fig:residualReduction-panda}
  }
  \caption{Performance of FMG cycles of our method
    with $\nu_1 = \nu_2 = 3$
    for tests on the unit square $(0,1)^2$, the rotated square $\Omega_r$,
    and the panda in \Cref{sec:squares,sec:panda}.
    The ordinates are the relative residuals/errors in $L^\infty$ norm,
    the abscissa is the iteration number of FMG cycles.
    The condition number of an FMG cycle is indicated
    by the smallest solution error that remains constant
    under more multigrid iterations.
  }
  \label{fig:residualReduction}
\end{figure}

\begin{table}
  \centering
  \caption{Averaged reduction rates of residuals in solving problems
    in \Cref{sec:squares,sec:squareMinusFlower,sec:squareMinusFourDisks,sec:panda}.}
  \begin{tabular}{|c|c|c|c|c|}
\hline
Tests  & cycles & \footnotesize $(\nu_1,\nu_2)=(2,1)$
  &\footnotesize $(\nu_1,\nu_2)=(2,2)$ & \footnotesize  $(\nu_1,\nu_2)=(3,3)$ \\ \hline
\multirow{2}{*}{the unit square $(0,1)^2$}                    & V     & 0.858                 & 0.410                 & 0.330                 \\ \cline{2-5} 
                                                & FMG   & 0.672                 & 0.246                 & 0.131                 \\ \hline
\multirow{2}{*}{the rotated square $\Omega_r$}                 & V     & 0.474                 & 0.266                 & 0.216                 \\ \cline{2-5} 
                                                & FMG   & 0.376                 & 0.162                 & 0.103                 \\ \hline
\multirow{2}{*}{$(0,1)^2$ minus a flower}          & V     & 0.384                 & 0.318                 & 0.221                 \\ \cline{2-5} 
                                                & FMG   & 0.296                 & 0.207                 & 0.108                 \\ \hline
  \multirow{2}{*}{\makecell[t]{$(0,1)^2$ minus four disks \\
  (a Dirichlet condition)}} & V     & 0.298                 & 0.236                 & 0.161                 \\ \cline{2-5} 
                                                & FMG   & 0.292                 & 0.120                 & 0.063                 \\ \hline
  \multirow{2}{*}{\makecell[t]{$(0,1)^2$ minus four disks \\
  (a Neumann condition)}} & V     & 0.373                 & 0.336                 & 0.279                 \\ \cline{2-5} 
                                                & FMG   & 0.347                 & 0.208                 & 0.135                 \\ \hline
  \multirow{2}{*}{\makecell[t]{panda with \\
  $(a,b,c)=(1,0,2)$}}   & V     & 0.652                 & 0.357                 & 0.178                 \\ \cline{2-5} 
                                                & FMG   & 0.548                 & 0.233                 & 0.119                 \\ \hline
  \multirow{2}{*}{\makecell[t]{panda with \\
  $(a,b,c)=(1,1,2)$}}   & V     & 0.661                 & 0.359                 & 0.189                 \\ \cline{2-5} 
                                                & FMG   & 0.636                 & 0.270                 & 0.126                 \\ \hline
\end{tabular}

  \label{tab:residualReduction}
\end{table}

In \Cref{fig:residualReduction},
 we show the performance of FMG cycles of our method
 in solving the tests in \Cref{sec:squares,sec:panda}.
For the unit and rotated squares,
 each FMG cycle respectively reduces the residual by a factor of
 7.6 and 9.7. 
As for the panda tests,
 each FMG cycle reduces the residual by a factor of
 8.4 and 7.9, respectively.
These reduction rates confirm the discussions on \Cref{tab:spectralRadius}
 in \Cref{sec:Vcycles}
 and are comparable to those
 of classical geometric multigrid methods. 
 
Finally in \Cref{tab:residualReduction},
 we list averaged reduction rates of our multigrid cycles
 for problems in
 \Cref{sec:squares,sec:squareMinusFlower,sec:squareMinusFourDisks,sec:panda}.
The improvement of FMG cycles over V-cycles
 are clearly demonstrated,
 verifying the claim in the ending sentence of \Cref{sec:Vcycles}. 
Altogether, \Cref{tab:timeConsumption,tab:residualReduction}
 imply that the choice of $(\nu_1,\nu_2)=(3,3)$
 is more cost-effective than those of $(\nu_1,\nu_2)=(2,1)$
 and $(\nu_1,\nu_2)=(2,2)$.
For complex domains
 and moderate grid sizes, 
 it might be appropriate to choose even greater values of $(\nu_1,\nu_2)$.



\section{Conclusions}
\label{sec:conclusions}

We have proposed a fourth-order cut-cell multigrid method
 for solving constant-coefficient elliptic equations
 on 2D irregular domains
 with the optimal complexity of $O(h^{-2})$. 
Based on the Yin space,
 our method is able to handle arbitrarily complex topology and geometry.
Results of comprehensive numerical tests
 demonstrate the accuracy, efficiency, robustness,
 and generality of the new method.

Prospects for future research are as follows.
First,
 this work motivates theoretical investigations
 on the effectiveness of the proposed multigrid method.
Second, 
 we will augment the proposed cut-cell method
 to elliptic equations with variable coefficients.
Lastly,
 we will follow the GePUP formulations in \cite{Zhang2016:GePUP,li25:_gepup_es}
 to develop a fourth-order INSE solver
 on irregular domains, 
 for which the proposed method in this work
 can be reused to solve pressure Poisson equations 
 and Helmholtz-like equations.



{\bf Acknowledgments.}
We acknowledge helpful comments from Shaozhen Cao, Lei Pang, and Chenhao Ye, 
 graduate students at the school of mathematical sciences
 in Zhejiang University.


\bibliographystyle{siamplain}
\bibliography{bib/ellipticFV2D}

\end{document}